\documentclass[12pt]{article}
\usepackage{amsmath,amssymb,amsthm,amsfonts,mathrsfs}
\usepackage[alphabetic, initials, lite]{amsrefs} 
\usepackage{appendix}
\usepackage{enumitem,graphics,xcolor,yfonts,colonequals}
\usepackage{stmaryrd}
\usepackage[alphabetic]{amsrefs}
\usepackage[all,cmtip]{xy}
\usepackage[colorlinks,anchorcolor=blue,citecolor=blue,linkcolor=blue,urlcolor=blue,bookmarksopenlevel=2,bookmarksopen=true]{hyperref}
\usepackage[margin=1in]{geometry}
\usepackage{fancyhdr}
\fancypagestyle{titlepage}
{
	\fancyhf{}

	\fancyfoot[l]{
	\href{https://mathscinet.ams.org/mathscinet/msc/msc2020.html}
	}
}

\AtBeginDocument{%
	\def\MR#1{}
}

\newcommand{\bN}{\mathbb{N}}    
\newcommand{\bZ}{\mathbb{Z}}    
\newcommand{\bQ}{\mathbb{Q}}    

\newcommand{\A}{\mathbb{A}}
\newcommand{\G}{\mathbb{G}}

\newcommand{\cC}{\mathcal{C}}
\newcommand{\cE}{\mathcal{E}}
\newcommand{\cJ}{\mathcal{J}}
\newcommand{\cLL}{\mathcal{L}}  
\newcommand{\cM}{\mathcal{M}}   
\newcommand{\cO}{\mathcal{O}}   

\newcommand{\bP}{\mathbb{P}}    

\newcommand{\Br}{\mathrm{Br}}
\newcommand{\inv}{\mathrm{inv}}
\newcommand{\Qbar}{\overline{\bQ}}

\newcommand{\kbar}{\overline{k}}

\newcommand{\Aut}{\operatorname{Aut}}   
\newcommand{\ord}{\operatorname{ord}}   
\newcommand{\Gal}{\operatorname{Gal}}	
\newcommand{\Pic}{\operatorname{Pic}}	
\newcommand{\Spec}{\operatorname{Spec}}
\newcommand{\GL}{\operatorname{GL}}
\newcommand{\Hom}{\operatorname{Hom}}
\newcommand{\Sym}{\operatorname{Sym}}
\newcommand{\im}{\operatorname{im}}
\newcommand{\rk}{\mathrm{rk}}

\newtheorem*{thm*}{Theorem}
\newtheorem*{prop*}{Proposition}
\newtheorem*{cor*}{Corollary}
\newtheorem{thm}{Theorem}[section]
\newtheorem{prop}[thm]{Proposition}
\newtheorem{cor}[thm]{Corollary}
\newtheorem{lemma}[thm]{Lemma}
\newtheorem{conj}[thm]{Conjecture}
\numberwithin{equation}{section}

\theoremstyle{definition}
\newtheorem{defn}[thm]{Definition}

\newtheorem{rmk}[thm]{Remark}

\newcommand{\injects}{\hookrightarrow}
\newcommand{\isom}{\simeq}

\begin{document}
\title{Uniform potential density for rational points on algebraic groups and elliptic K3 surfaces}
\author{Kuan-Wen~Lai \and Masahiro~Nakahara}
\date{}

\newcommand{\ContactInfo}{{
\bigskip\footnotesize

\bigskip
\noindent K.-W.~Lai,
\textsc{Department of Mathematics \& Statistics\\
University of Massachusetts\\
Amherst, MA 01003, USA}\par\nopagebreak
\noindent\texttt{lai@math.umass.edu}

\bigskip
\noindent M.~Nakahara,
\textsc{Department of Mathematics\\
University of Washington\\
Seattle, WA 98195, USA}\par\nopagebreak
\noindent\texttt{mn75@uw.edu}
}}

\maketitle
\thispagestyle{titlepage}

\begin{abstract}
A collection of varieties satisfies uniform potential density if each of them contains a dense subset of rational points after extending its ground field by a bounded degree. In this paper, we prove that uniform potential density holds for connected algebraic groups of a fixed dimension over fields of characteristic zero as well as elliptic K3 surfaces over number fields.
\end{abstract}

\setcounter{tocdepth}{2}
\tableofcontents

\section{Introduction}
\label{sect:intro}

A variety $X$ over a field $k$ satisfies \emph{potential density} if there exists a finite extension $L/k$ such that the set of $L$-rational points $X(L)$ is Zariski dense. For example, abelian varieties over number fields satisfy potential density \cite{HT00}*{Proposition~3.1}, while rational points do not have to be dense over the original field. Bogomolov and Tschinkel \cite{BT00} proved potential density for K3 surfaces $X$ over number fields with an elliptic fibration, though it is not known whether $X(k)$ is dense as soon as it is non-empty. In this paper, we consider whether the degree of the extension $L/k$ can be taken uniformly with respect to the geometric properties of $X$:

\begin{defn}
A collection $\cC$ of varieties satisfies \emph{uniform potential density} if there exists a constant $d_\cC$ such that for any $X/k\in \cC$, there exists a finite extension $L/k$ of degree bounded by $d_\cC$ such that $X(L)$ is Zariski dense.
\end{defn}

We allow the members in $\cC$ to have different ground fields $k$. Also, the constant $d_\cC$ is independent of the ground field. Examples of collections of varieties that satisfy uniform potential density include smooth curves of genus zero (they can be realized as plane conics and so the rational points are dense after a quadratic extension), del Pezzo surfaces (see Proposition \ref{prop:upd_delPezzo}), and torsors under simple linear algebraic groups of a fixed dimension \cite{Tit92}. Our first main theorem is based on the last example:

\begin{thm}
\label{thm:main_algebraicGroup}
Uniform potential density holds for
\begin{itemize}
    \item the collection of torsors under connected linear algebraic groups of a fixed dimension over infinite perfect fields;
    \item the collection of connected algebraic groups of a fixed dimension over fields of characteristic zero.
\end{itemize}

\end{thm}

In particular, the second part of the theorem implies that uniform potential density holds for abelian varieties of a fixed dimension. We also compute an explicit bound $d_{\cC}$ for the degree of the field extension in the case of abelian varieties (see Theorem~\ref{thm:nondeg_abelianVariety}). To prove Theorem~\ref{thm:main_algebraicGroup}, we use Chevalley's structure theorem to divide the problem into two parts: uniform potential density for (1) abelian varieties (Corollary~\ref{cor:UPD_abelianVariety}), and for (2) torsors under connected linear algebraic groups (Corollary~\ref{cor:UPD_torsor_linearAlgGp}). For abelian varieties, we use Zarhin's trick to first embed an abelian variety into a projective space with bounded degree. Then we use the Mordell--Lang conjecture to construct a non-torsion point over a bounded extension of the ground field. From there we use an inductive argument while keeping track of degrees of extensions to construct a rational point that generates a dense subgroup.

For torsors under connected linear algebraic groups, we start with a result of \cite{Tit92} that concerns the splitting fields of geometrically almost simple groups. More precisely, he showed that the degree of field extension to make such a group split can be taken to divide an integer that depends only the type of the group. We then show that one can deduce the uniform potential density for torsors under geometrically almost simple groups, and then to semisimple groups, and finally to connected linear algebraic groups.

Our second main theorem extends the result for elliptic K3 surfaces by Bogomolov and Tschinkel \cite{BT00}.

\begin{thm}
\label{thm:main_ellipticK3}
Uniform potential density holds for the collection of elliptic K3 surfaces over number fields.
\end{thm}

By an elliptic K3 surface, we mean a K3 surface together with a morphism $X\to \bP^1$ whose generic fiber is a smooth curve of genus one over the function field $k(\bP^1)$. Note that the analogous statement for curves of genus one fails (Proposition~\ref{prop:noUPD_genus-one-curve}). We also compute explicitly the constant $d_\cC$ that one can take for potential density to hold. See Theorem~\ref{thm:ellk3density} for details. Our proof of Theorem~\ref{thm:main_ellipticK3} follows that of potential density in \cite{BT00}. While some of their constructions can be done over the ground field, others require an extension whose degree is unknown. We begin by using a result of \cites{Min87} on finite subgroups of $\GL(n,\bZ)$ to show that the geometric N\'eron--Severi group can be defined over a bounded extension of the ground field. The key problem is then to bound the degree of the field of definition of an elliptic or rational multisection which is used to produce rational points. This is done using lattice theory in Lemma~\ref{lemma:existenceRatMult}.

Theorem~\ref{thm:main_ellipticK3} provides an evidence to the following conjectures in the case of K3 surfaces: let $\cC=\{(X,k)\}$ be a collection of smooth and geometrically integral varieties $X$ over number fields $k$. Let $\Br(X)$ denote the Brauer group of $X$ and define
$$
    \Br_0(X)\colonequals\im(\Br(k)\longrightarrow\Br(X)).
$$
We say $\cC$ satisfies
\begin{itemize}
    \item \emph{uniform boundedness of exponent for Brauer groups (UBEB)} if, for every integer $d>0$, there exists a constant $c_d$ such that $c_d\alpha=0$ for any $\alpha\in\Br(X_L)/\Br_0(X_L)$ where $X/k\in\cC$ and $L/k$ is an extension such that $[L:k]\leq d$;
    \item \emph{Brauer--Manin obstruction is the only one to weak approximation (BMWA)} if for any $X/k\in\cC$ and finite extension $L/k$, we have $X(L)$ dense in $X(\A_L)^{\Br}$.
    \item \emph{uniform potential local solubility (UPLS)} if, there exists a constant $c$ such that for any $X/k\in \cC$, there exists an extension $L/k$ of degree $\leq c$ such that $X_L$ is everywhere locally soluble.
\end{itemize}

\begin{thm}
Let $\cC$ be a collection of smooth projective geometrically integral varieties $X/k$ over number fields such that $H^1(X,\cO_X)=0$ and $\Pic(X_{\kbar})$ is torsion-free of bounded rank. If $\cC$ satisfies UBEB, BMWA, and UPLS then it also satisfies uniform potential density.
\end{thm}

For K3 surfaces over number fields, both UBEB and BMWA are conjectures \cites{VA17,Sko09}, see \S\ref{sect:BrauerUPD} for a discussion. It would follow from these conjectures that uniform potential density holds for all K3 surfaces over number fields provided that UPLS holds. Our Theorem~\ref{thm:main_ellipticK3} can be taken as further evidence toward these conjectures.

\bigskip
\noindent
\textbf{Organization of the paper.}
In \S\ref{sect:finiteSubgpGLn}, we recall a result due to Minkowski which plays a fundamental role in proving the main theorems and discuss the uniform potential density for del Pezzo surfaces. We prove uniform potential density for algebraic groups in \S\ref{sect:UPD_abelianVariety} and for elliptic K3 surfaces in \S\ref{sect:UPD_ellipticK3}. The relationship between uniform potential density and the Brauer--Manin obstruction is discussed in \S\ref{sect:BrauerUPD}.

\bigskip
\noindent
\textbf{Acknowledgements.}
We thank Zinovy Reichstein for his kind help for
dealing with the case of algebraic groups, and thank Brendan Hassett for his suggestion on improving the proof of the case of elliptic K3 surfaces. We also thank Daniel Loughran for suggesting us including del Pezzo surfaces for uniform potential density and for his comments on the first draft of the paper. We would also like to thank the referees for correcting errors in earlier versions of the paper and giving suggestions on improving the expositions. The second author was sponsored by EPSRC grant EP/R021422/2.

\section{Finite subgroups of general linear groups}
\label{sect:finiteSubgpGLn}

We begin with the following well-known result due to Minkowski.

\begin{thm}[\cite{Min87}]
\label{thm:minkowski}
For every positive integer $n$, there are only finitely many isomorphism classes of finite subgroups in $\GL(n,\bZ)$.
\end{thm}

See also \cites{KP02} for a discussion on this theorem. Motivated by this result, we make the following convenient notation which will be used throughout the paper.

\begin{defn}
\label{defn:c_n}
For every positive integer $n$, we define $\mathbf{c}_n$ to be the maximal cardinality among the cardinalities of the finite subgroups in $\GL(n,\bZ)$.
\end{defn}

Notice that the constant $\mathbf{c}_n$ increases as $n$ increases since $\GL(n,\bZ)$ is a subgroup of $\GL(n+1,\bZ)$. Theorem \ref{thm:minkowski} will be critical in proving uniform potential density for many cases later on. As a quick application, we prove that del Pezzo surfaces satisfy uniform potential density. If $X$ is a del Pezzo surface of degree $1\leq d\leq 7$, then the Galois action on $\Pic(X_{k^s})\isom\bZ^{10-d}$ gives rise to a homomorphism
\begin{equation}
\label{eqn:factor_through_GL}
    \Gal(k^s/k)\longrightarrow\GL(10-d,\bZ).
\end{equation}
Because the Picard group is generated by the finitely many $(-1)$-curves on $X_{k^s}$, there exists a finite extension $L/k$ such that $\Pic(X_L) = \Pic(X_{k^s})$ and $\Gal(k^s/L)$ appears as the kernel of \eqref{eqn:factor_through_GL}. In particular, the group $\Gal(L/k)\cong\Gal(k^s/k)/\Gal(k^s/L)$ is isomorphic to a subgroup of $\GL(10-d,\bZ)$. Thus
$$
    [L:k]
    = |\Gal(L/k)|
    \leq\mathbf{c}_{10-d}
    \leq\mathbf{c}_{9}.
$$
The condition $\Pic(X_L) = \Pic(X_{k^s})$ implies that every $(-1)$-curve of $X_{k^s}$ is defined over $L$. Contracting part of them gives a birational morphism $X_L\to\bP^2_L$, whose inverse produces a dense subset of rational points on $X_L$.

Using the known properties of del Pezzo surfaces however, the bound for the degree $L/k$ in which we get a dense set of rational points can be vastly improved from $\mathbf{c}_9$.

\begin{prop}
\label{prop:upd_delPezzo}
For a del Pezzo surface $X$ of degree $1\leq d\leq 9$ over an infinite field $k$, we can obtain a Zariski dense subset of rational points by passing to a finite extension $L/k$ of degree bounded by a number depending only on $d$ as listed below:
\begin{center}
\begin{tabular}{|c|c|c|c|c|c|c|c|c|c|}
   \hline 
   $d$ & $1$ & $2$ & $3$ & $4$ & $5$ & $6$ & $7$ & $8$ & $9$ \\
   \hline 
   $[L:k]\leq$ & $240\cdot 4$ & $4$ & $3$ & $4$ & $1$ & $6$ & $1$ & $2$ & $3$ \\
   \hline
\end{tabular}    
\end{center}
\end{prop}

\begin{proof}
We divide the proof into cases based on $d$.

Case~$d=9$: in this case, $X$ is a Sever--Brauer surface, which corresponds to (the isomorphism class of) a central simple algebra $A$ of degree $3$ \cite{GS17}*{Theorem~5.2.1}. By Wedderburn's theorem (see, e.g., \cite{GS17}*{Theorem~2.1.3}), $A$ is isomorphic to a matrix ring $M_n(D)$ for some $n\geq 1$ and some division algebra $D$. Notice that the degree of $D$ divides $3$ in our situation. By \cite{GS17}*{Theorem~2.2.7}, the algebra $D$, and thus $A$, admits a splitting field $L/k$ of degree at most $3$. It follows that $X_L\cong\bP^2_L$ and so $X(L)$ is dense.

Case~$d=8$: there are two possibilities, $X$ is isomorphic to either a blow up of $\bP^2$ at a point or a quadric surface. In the former case $X$ is birational to $\bP^2$. In the latter case, we can find a rational point on $X$ by taking a quadratic extension $L/k$, in which case the stereographic projection from the point maps $X_L$ birationally onto $\bP^2_L$.

Case~$d=5,7$: $X(k)$ is dense since $X$ is birational to $\bP^2$, see \cite{VA13}*{Theorem~2.1}.

Case~$d=6$: there are a total of six $(-1)$-curves on $X_{k^s}$, so one of them is defined over an extension $L/k$ of degree at most $6$. Contracting this $(-1)$-curve gives a del Pezzo surface of degree $7$ over $L$, which shows $X(L)$ is dense by Case~$d=7$.

Case~$d=4$: if $X$ has a rational point, then it is unirational and thus contains a dense subset of rational points \cite{Pie12}*{Proposition~5.19}. Since $X$ is isomorphic to the intersection of two quadrics in $\bP^4$, it has a point of degree at most $4$.

Case~$d=3$: the situation is the same as $d=4$, namely, the existence of a rational point is sufficient \cite{Kol02}*{Theorem~1.1}. In this case, $X$ is isomorphic to a cubic surface in $\bP^3$, so the bound can be taken as $3$.

Case~$d=2$: the linear system $|-K_X|$ defines a double cover $X\to\bP^2$, so there exists a rational point over a quadratic extension $L/k$. By \cite{STVA14}*{Corollary~1.3}, $X_L$ becomes unirational over a further quadratic extension $L'/L$, so the bound in this case is $4$.

Case~$d=1$: the number of $(-1)$-curves on $X_{k^s}$ equals $240$, so there is an extension $L/k$ of degree at most $240$ where one of them is defined over $L$. Contracting this $(-1)$-curve gives a degree $2$ del Pezzo surface over $L$. By Case $d=2$, we can get a dense set of rational points over a further extension $L'/L$ of degree at most $4$.
\end{proof}

\begin{rmk}
If $\operatorname{char}(k)\neq 2,3$, the bound for $d=1$ in Proposition~\ref{prop:upd_delPezzo} can be slightly improved to $240$. To see this, first recall that the linear system $|-K_X|$ induces an elliptic fibration $\mathcal{E}\to\bP^1$ with a section. When $\operatorname{char}(k)\neq 2,3$, the linear system $|-3K_X|$ embeds $X$ into the weighted projective space $\bP(2,3,1,1)$ with coordinates $x,y,z,w$ as a sextic hypersurface
$
    y^2 = x^3 + f(z,w)x + g(z,w),
$
where the fibration $\mathcal{E}$ is given explicitly by the projection (see, e.g., \cite{VA13}*{\S1.5.4} or \cite{SvL14}*{\S1})
$$
    \bP(2,3,1,1)\dashrightarrow\bP^1
    :[x:y:z:w]\mapsto [z:w].
$$
Under this model, Salgado and van Luijk proved that all the $(-1)$-curves on $X_{k^s}$ appear as sections of $\mathcal{E}_{k^s}$ \cite{SvL14}*{Lemma~2.2} and none of them is torsion \cite{SvL14}*{Theorem~6.4}. Let $L/k$ be an extension of degree at most $240$ over which one of the $(-1)$-curves $C$ is defined. Then the orbit of $C$ under the group law of the generic fiber produces a dense subset of sections and thus a dense subset of rational points.
\end{rmk}

\section{Uniform potential density for algebraic groups}
\label{sect:UPD_abelianVariety}

In this section, we establish uniform potential density for connected algebraic groups over fields of characteristic zero. We start by studying basic properties about rational points on elliptic curves, then use it to establish uniform potential density for abelian varieties. We then prove uniform potential density for torsors under connected linear algebraic groups. We finish proof of Theorem \ref{thm:main_algebraicGroup} by combining these two results via Chevalley's structure theorem.

Most of the results in this section are proved over fields of characteristic zero, but some statements hold over more general fields.

\subsection{Genus one curves versus elliptic curves}
\label{subsect:UPD_ellipticCurve}

Uniform potential density fails for curves of genus one, however, it holds for elliptic curves, that is, genus one curves that contains a rational point. Let us explain these two facts below.

\begin{prop}
\label{prop:noUPD_genus-one-curve}
The collection of genus one curves over number fields does not satisfy uniform potential density.
\end{prop}

\begin{proof}
Fix an elliptic curve $E$ over a number field $k$. Recall that the genus one curves over $k$ which are isomorphic to $E$ over the algebraic closure $\kbar$ are classified by the Weil--Chatelet group $H^1(k,E)$. For every genus one curve $C$ over $k$, its index $i_C$ is defined as the greatest common divisor of the degrees of closed points on $C$. It is well known that $H^1(k,E)$ is torsion and that the order of $[C]\in H^1(k,E)$ divides $i_C$.

For any integer $n$, the Kummer sequence of $\Gal(\kbar/k)$-modules
$$\xymatrix{
    0\ar[r]
    & E[n]\ar[r]
    & E\ar[r]^-{\times n}
    & E\ar[r]
    & 0
}$$
induces the exact sequence
$$\xymatrix{
    E(k)/n(E(k))\ar[r]
    & H^1(k,E[n])\ar[r]
    & H^1(k,E)[n]
}$$
The group $H^1(k,E[n])$ is infinite. (See Lemma~\ref{lemma:hilbertianGalois}.) Since $E(k)/nE(k)$ is finite by the weak Mordell--Weil theorem \cite{Sil09}*{VIII, Theorem~6.7}, the group $H^1(k,E)[n]$ is infinite. In particular, there exists $[C]\in H^1(k,E)$ of arbitrarily large prime order and so arbitrarily large index. Such a curve $C$ does not have rational points over fields of degree smaller than the index over $k$. Thus the collection of genus one curves does not satisfy uniform potential density.
\end{proof}

Given an abelian group $G$ and a subgroup $H$, we define the saturation of $H$ (in $G$) to be
$$
    \{
        g\in G : \textnormal{there exists } n\in\bZ\textnormal{ such that }ng\in H
    \}.
$$
Uniform potential density for elliptic curves over number fields can be derived from the following more general statement:

\begin{thm}
\label{thm:gamma_ellipticCurve}
Let $E$ be an elliptic curve over a field $k$ of characteristic $0$. Suppose that $\Gamma\subset E(\kbar)$ is a finitely generated subgroup and let $\Gamma_0\subset E(\kbar)$ be its saturation. Then there exists an extension $L/k$ of degree at most $18$ such that $E(L)\not\subset\Gamma_0$. Moreover, if $k$ is a number field, the degree bound can be improved to $2$.
\end{thm}

Our approaches over number fields and over fields of characteristic zero are different. Let us start with the case of number fields. The general case will be proved later.

\begin{proof}[Proof of Theorem~\ref{thm:gamma_ellipticCurve} over number fields]
Because $\Gamma$ is finitely generated, there exists a finite extension $L_0/k$ such that $\Gamma_0$ is contained in the saturation of $E(L_0)$.

Let $O$ be the origin of $E$ and let $f\colon E\to\bP^1$ be the double cover defined by $|2O|$. By Hilbert's irreducibility theorem, for any finite extension $L/k$, there exists infinitely many $x\in\bP^1(k)$ such that the fiber $f^{-1}(x)$ is irreducible of degree $2$ over $L$.
Using this fact, we can inductively produce a list of $k$-rational points
$
    \{x_1, x_2, x_3, \dots\}\subset\bP^1(k)
$
such that the splitting fields $L_i$ of the fibers $f^{-1}(x_i)$ satisfy the property:

\vspace{-5pt}
\begin{quote}\it
    For every integer $i\geq 1$, the fiber $f^{-1}(x_i)$ is irreducible over the composite field $L_0L_1\cdots L_{i-1}$.
\end{quote}
\vspace{-5pt}

\noindent As each $L_i/k$, $i\geq 1$, is a quadratic extension, it follows that $L_i\cap L_j = k$ whenever $i,j\geq 0$ are distinct.

For each $x_i\in\bP^1(k)$, the fiber $f^{-1}(x_i)$ splits into two points $y_i, y_i'$ over $L_i$ such that $y_i=-y_i'$ in the group law of $E$. Note that $y_i$ and $y_i'$ are $\Gal(\kbar/k)$-conjugates. Suppose $y_i\in\Gamma_0$. Then there exists $n>0$ such that $ny_i\in E(L_0)$. Because $ny_i$ is defined over $L_i$ and $L_i\cap L_0=k$, the point $ny_i$ is in fact defined over $k$. Since $ny_i$ and $ny_i'$ are $\Gal(\kbar/k)$-conjugates, this implies $ny_i=ny_i'$ and thus $2ny_i=0$. In particular, $y_i\in E(L_i)_\mathrm{tor}$.
By Merel's theorem (see, e.g., \cite{Sil09}*{VIII, Theorem~7.5.1}), there exists an integer $N$ such that $|E(L)_\mathrm{tor}|\leq N$ for any quadratic extension $L/k$. In particular, every point in $E(L_i)_\mathrm{tor}$ where $i\geq1$ has order at most $N$. It follows that
\begin{equation}
\label{eqn:torsionOverQuadraticExt}
    \bigcup_{i\in\bN}E(L_i)_\mathrm{tor}\subset E(\kbar)[N!]
\end{equation}
where $E(\kbar)[N!]$ is the finite set of torsion $\kbar$-points of order dividing $N!$. If $y_i\in\Gamma_0$ for infinitely many $i\in\bN$, we would obtain infinitely many points on the left hand side of \eqref{eqn:torsionOverQuadraticExt}, which is impossible. Therefore, all but finitely many $y_i$ are not in $\Gamma_0$. Hence all but finitely many of the quadratic extensions $L_i/k$ satisfies $E(L_i)\not\subset\Gamma_0$.
\end{proof}

\begin{proof}[Proof of Theorem~\ref{thm:gamma_ellipticCurve} in the general case]
Let us consider the abelian surface $X\colonequals E\times E$ and the divisor $D\colonequals E\times\{0\}+\{0\}\times E$ on $X$. Note that $D^2 = 2$, so $D$ is ample by \cite{Kan94}*{Corollary~2.2}. This implies that $3D$ is very ample \cite{EGM20}*{(2.26)}, so we can embed $X$ into $\bP^n$ as a subvariety of degree $(3D)^2=18$. Bertini's theorem guarantees the existence of a smooth sectional curve $C\subset X$ of degree $18$. One can compute that $C$ has genus $10$ by the adjunction formula.

Define $\Lambda_0\colonequals\Gamma_0\times\Gamma_0\subset X(\kbar)$. According to the Mordell--Lang conjecture (see, e.g., \cite{Maz00} and the references cited there), the intersection $C\cap\Lambda_0$ contains only finitely many points. It follows that, among the infinitely many closed points of degree at most $18$ on $C$ cut out by hyperplanes defined over $k$, there exists at least one not in $\Lambda_0$.
As a result, there exists an extension $L/k$ of degree at most $18$ such that $X_L\cong E_L\times E_L$ has an $L$-rational point $p\notin\Lambda_0$. By projecting $p$ to one of the copies of $E_L$, we obtain an $L$-rational point on $E_L$ not lying in $\Gamma_0$.
\end{proof}

Theorem \ref{thm:gamma_ellipticCurve} has an immediate consequence for uniform potential density for elliptic curves, which is well-known to experts but we include the corollary here for completeness.

\begin{cor}
\label{cor:UPD_ellipticCurve}
Let $E$ be an elliptic curve over a field $k$ of characteristic zero. Then there exists an extension $L/k$ of degree bounded by $18$ such that $E(L)$ contains a non-torsion point and thus is Zariski dense in $E_L$. The bound $18$ can be improved to $2$ if $k$ is a number field. As a consequence, the collection of elliptic curves over fields of characteristic zero satisfies uniform potential density.
\end{cor}

\begin{proof}
Applying Theorem~\ref{thm:gamma_ellipticCurve} with $\Gamma = \{0\}$, one obtains $p\in E(L)$ not lying in $\Gamma_0$ for an extension $L/k$ of degree bounded by $2$ in the number field case and by $18$ in general. Because $\Gamma_0$ contains all torsion points of $E(\kbar)$, the point $p$ is non-torsion. Hence it generates a subset $\{np\mid n\in\bZ\}\subset E(L)$ Zariski dense in $E_L$.
\end{proof}

Let us introduce another corollary of Theorem~\ref{thm:gamma_ellipticCurve}. Recall that an elliptic surface $\mathcal{E}\to C$ over a field $k$ \emph{splits} if there exist an elliptic curve $E_0$ over $k$ and a birational map between elliptic fibrations
$\xymatrix{
    \mathcal{E}\ar@{-->}[r]^-\sim & E_0\times C.
}$
Suppose that $\mathcal{E}$ does \emph{not} split. Then the group $E(k(C))$ is finitely generated by the Mordell--Weil theorem for function fields \cite{Sil94}*{III, Theorem~6.1}.

\begin{cor}
\label{cor:rankup_ellipticCurve}
Suppose that $E$ is an elliptic curve over $k$, where $k$ is
\begin{itemize}
    \item a number field, or
    \item the function field of a curve $C$ over a field of characteristic $0$ such that the associated elliptic surface $\mathcal{E}\to C$ does not split.
\end{itemize}
Then there is an extension $L/k$ of degree at most $2$ in the number field case and at most $18$ in the function field case such that $\rk(E(L)) > \rk(E(k))$.
\end{cor}

\begin{proof}
By the Mordell--Weil theorem for number fields or for function fields, where the latter requires $\mathcal{E}\to C$ to be non-split, the group $E(k)$ is finitely generated. Applying Theorem~\ref{thm:gamma_ellipticCurve} with $\Gamma = E(k)$, we conclude that there exists an extension $L/k$ and $p\in E(L)$ such that $p\notin\Gamma_0$. Moreover, the degree of $L/k$ can be taken to be at most $2$ in the number field case and at most $18$ in the function field case. Since $p$ does not lie in the saturation of $E(k)$, we have $\rk(E(L))>\rk(E(k))$.
\end{proof}

\subsection{Abelian varieties of higher dimensions}
\label{subsect:UPD_abelianVariety}

Our proof of unifrom potential density for abelian varieties over a fixed dimension, follows the main ideas for potential density in \cite{HT00}*{Proposition~3.1}.

\subsubsection{Existence of a non-torsion point}
\label{subsubsect:exist_nontorsion}

\begin{lemma}
\label{lemma:embedAbelianVariety}
Let $X$ be a principally polarized abelian variety of dimension $n$ over a field $k$. Then $X$ admits a very ample line bundle $\mathcal{D}$ that embeds $X$ as a subvariety in $\bP^{2^n-1}$ of degree $6^n(n!)$.
\end{lemma}

\begin{proof}
For every $\cE\in\Pic(X_{\kbar})$, we will denote by $\varphi_{\cE}$ the isogeny
$$\xymatrix{
    \varphi_{\cE}\colon X\ar[r] &
    X^\vee,\quad x\ar@{|->}[r] & t_x^*\cE\otimes\cE^{-1}
}$$
where $t_x$ is the translation by $x$. Let $\lambda\colon X\longrightarrow X^\vee$ be a principal polarization and $\mathcal{P}$ be the Poincare bundle on $X\times X^\vee$. Then there exists $\mathcal{L}\in\Pic(X_{\kbar})$ such that $\lambda=\varphi_\mathcal{L}$. Moreover, the line bundle
$
    \mathcal{D}\colonequals
    (\mathrm{id}_X,\lambda)^*\mathcal{P}
    \in\Pic(X)
$
is ample and satisfies \cite{EGM20}*{Proposition~11.1}
$$
    \varphi_\mathcal{D}
    =2\lambda
    =2\varphi_\mathcal{L}
    \in\Hom^{\Sym}(X,X^\vee).
$$
Using the exact sequence
$
    0\to X^\vee(k)\to\Pic(X)\xrightarrow{\varphi} \Hom^{\Sym}(X,X^\vee),
$
we conclude that
\begin{equation}
\label{eqn:M=L2N}
    \mathcal{D}=\mathcal{L}^{\otimes2}\otimes\mathcal{N}
\end{equation}
for some $\mathcal{N}\in\Pic^0(X)=X^\vee(k)$.

The line bundle $3\mathcal{D}$ is very ample as $\mathcal{D}$ is ample \cite{EGM20}*{(2.26)} and thus defines an embedding $X\hookrightarrow\bP^\ell$. To compute the degree of $X$ in $\bP^\ell$, we use \eqref{eqn:M=L2N} and the Riemann--Roch formula for abelian varieties to get
\begin{equation}
\label{eqn:DinNA}
    \chi(\mathcal{D})
    =c_1(\mathcal{D})^n/n!
    =(2c_1(\mathcal{L}) + c_1(\mathcal{N}))^n/n!.
\end{equation}
We have $c_1(\mathcal{N})=0$ and also
$
    c_1(\mathcal{L})^n/n!
    =\chi(\mathcal{L})
    =\sqrt{\deg\lambda}=1.
$
Plugging these into \eqref{eqn:DinNA}, we obtain $\chi(\mathcal{D})=2^n$. Therefore, $\mathcal{D}$ embeds $X$ into $\bP^{2^n-1}$ with degree given by
$$
    \deg(X)
    = (3c_1(\mathcal{D}))^n
    = 3^n\chi(\mathcal{D})(n!)
    = 3^n2^n(n!)
    = 6^n(n!).
$$
\end{proof}

The following lemma is an analog of Theorem~\ref{thm:gamma_ellipticCurve} for abelian varieties.

\begin{lemma}
\label{lemma:Mordell-Lang}
Let $A$ be an abelian variety of dimension $g\geq 2$ over a field $k$ of characteristic zero. Suppose that $\Gamma\subset A(\kbar)$ is a finitely generated subgroup and let $\Gamma_0$ be its saturation. Then there exists an extension $L/k$ of degree at most $6^{8g}\cdot (8g)!$ such that $A(L)\not\subset\Gamma_0$.
\end{lemma}

\begin{proof}
Fix an isogeny $\varphi\colon A\longrightarrow A^\vee$ and define $X\colonequals A^4\times (A^\vee)^4$. Consider the finitely generated subgroup
$$
    \Gamma_X\colonequals
    \Gamma_0^4\times\varphi(\Gamma_0)^4\subset X(\kbar).
$$
and denote by $(\Gamma_X)_0\subset X(\kbar)$ its saturation. Let us first prove that there exists a finite extension $L/k$ of degree at most $6^{8g}\cdot(8g)!$ such that $X(L)\not\subset(\Gamma_X)_0$. In the case that $X$ contains an elliptic curve $E$ as a subgroup over $k$, we can apply Theorem~\ref{thm:gamma_ellipticCurve} to $E$ and the subgroup $E(\kbar)\cap(\Gamma_X)_0$. This gives us an extension $L/k$ of degree at most $18<6^{8g}\cdot (8g)!$ such that $E(L)$ is not contained in the saturation of $E(\kbar)\cap(\Gamma_X)_0$ in $E(\kbar)$, which implies that $X(L)\not\subset(\Gamma_X)_0$, as desired. Therefore, we assume that $X$ does not contain an elliptic curve over $k$ below.

Due to Zarhin's trick \cite{EGM20}*{Theorem~11.29}, the product $X=A^4\times (A^\vee)^4$ admits a principal polarization. Note that $X$ has dimension $8g$, so Lemma~\ref{lemma:embedAbelianVariety} implies that $X$ can be realized as a projective variety of degree $6^{8g}\cdot (8g)!$. Let us cut $X$ with hyperplanes over $k$ subsequently to get a curve $C\subset X$, which has genus at least $2$ since $X$ contains no elliptic curve by hypothesis. Moreover, we may assume that $C$ is smooth and passes through the origin of $X$ due to Bertini's theorem. In this setting, the intersection $C(\kbar)\cap(\Gamma_X)_0$ is finite due to the Mordell--Lang conjecture. (See, e.g., \cite{Maz00}.) On the other hand, we can obtain infinitely many points on $C$ of degrees at most $\deg(C) = 6^{8g}\cdot (8g)!$ by further cutting $C$ with hyperplanes over $k$. Hence there exists a point $p\in C$ not in $(\Gamma_X)_0$ whose residue field $L\colonequals k(p)$ satisfies $[L:k]\leq 6^{8g}\cdot (8g)!$. In particular, $X(L)\not\subset(\Gamma_X)_0$.

Finally, take any $p\in X(L)$ not lying in $(\Gamma_X)_0$. If all the projections $X\to A$ and $X\to A^\vee$ map $p$ to $\Gamma_0$ and $\varphi(\Gamma_0)$, respectively, then $p$ belongs to $\Gamma_X$, a contradiction. Hence there exists a projection $X\to A$ or $X\to A^\vee$ that maps $p$ to an $L$-point $q$ not in $\Gamma_0$ or not in $\varphi(\Gamma_0)$, respectively. In the former case, we are done. In the latter case, the dual isogeny $\varphi^\vee\colon A^\vee\to A$ maps $q$ to $\varphi^\vee(q)\in A(L)$ not lying in $\Gamma_0$, which completes the proof.
\end{proof}

\begin{prop}
\label{prop:UPD_simpleAbelianVariety}
Let $A$ be an abelian variety of dimension $g$ over a field $k$ of characteristic zero. Then there exists an extension $L/k$ of degree bounded by a constant $d_g$ as given below such that $A(L)$ contains a non-torsion point:
\begin{itemize}
    \item if $g=1$, we can let $d_g = 2$ when $k$ is a number field and $d_g = 18$ in general.
    \item if $g\geq 2$, we can let $d_g = 6^{8g}\cdot(8g)!$.
\end{itemize}
In particular, $A(L)$ is dense in $A_L$ provided that $A$ is geometrically simple.
\end{prop}

\begin{proof}
The cases $g=1$ and $g\geq 2$ follow respectively from Theorem~\ref{thm:gamma_ellipticCurve} and Lemma~\ref{lemma:Mordell-Lang} by taking $\Gamma = \{0\}$. If $A$ is geometrically simple, then Faltings' theorem on subvarieties of abelian varieties (see \cite{Maz00}*{Theorem~3.1}) says that the Zariski closure of the subgroup generated by $p$ in $A_L$ coincides with $A_L$. Hence $A(L)$ is dense in $A_L$.
\end{proof}

\subsubsection{Existence of a nondegenerate point}
\label{subsubsect:exist_nondegenerate}

\begin{defn}
For an abelian variety $A$ over a field $k$, we say a rational point $p\in A(k)$ is \emph{nondegenerate} if the subgroup generated by $p$ is Zariski dense in $A$.
\end{defn}

\begin{lemma}
\label{lemma:dimensionGrowth}
Let $A$ and $A'$ be abelian varieties over a field $k$ of characteristic zero and let $g'\colonequals\dim(A')$. Suppose that $A$ contains a nondegenerate point $p\in A(k)$. Then there exists an extension $L/k$ of degree bounded by $6^{8g'}\cdot (8g')!$ together with a point $p'\in A'(L)$ such that the Zariski closure
\[
    B\colonequals\overline{\bZ(p,p')}\subset (A\times A')_L
\]
is an abelian subvariety with $\dim(B)>\dim(A)$.
\end{lemma}

\begin{proof}
The group of homomorphisms $\Hom(A,A')$ is finitely generated as a module over $\bZ$, so the image of $p$ under $\Hom(A,A')$ forms a finitely generated subgroup
$$
    \Gamma\colonequals\Hom(A,A')(p)\subset A'(\overline{k}).
$$
When $g'\geq 2$ (resp. $g'=1$), Lemma~\ref{lemma:Mordell-Lang} (resp. Theorem~\ref{thm:gamma_ellipticCurve}) implies that there exists an extension $L/k$ of degree bounded by $6^{8g'}\cdot (8g')!$ (resp. at most $18$) such that $A'(L)$ contains a point $p'$ not in the saturation of $\Gamma$. In any case, it is sufficient to take $6^{8g'}\cdot (8g')!$ as the bound. Let
$
    q\colonequals(p,p')\in (A\times A')(L)
$
and denote by $B\subset (A\times A')_{L}$ the abelian subvariety obtained as the Zariski closure of $\bZ q$.

Let us prove that $\dim(B) > \dim(A)$ by adapting part of the proof of \cite{HT00}*{Lemma~3.3}. Assume, to the contrary, that $\dim(B) = \dim(A)$. In this case, the left projection
$$\xymatrix{
    \pi\colon A\times A'\ar[r] & A
}$$
restricts to an isogeny $B\longrightarrow A$ over $L$, so we can regard $B$ as an element in the group of homomorphisms up to isogeny:
\[
    \beta
    \in\Hom^0(A,A')
    \colonequals\Hom(A,A')\otimes\bQ.
\]
In particular, there exists a nonzero integer $d$ such that $(d\beta)(p) = dp'$. This implies that $p'$ lives in the saturation of $\Gamma$, contradiction. This proves that $\dim(B)>\dim(A)$.
\end{proof}

\begin{thm}
\label{thm:nondeg_abelianVariety}
Let $A$ be an abelian variety of dimension $g$ over a field $k$ of characteristic zero. Then there exists a finite extension $L/k$ of degree bounded by the constant
$$
    6^{4g(g+1)}\prod_{r=1}^{g} (8r)!
$$
such that $A(L)$ contains a nondegenerate point.
\end{thm}

\begin{proof}
By Proposition~\ref{prop:UPD_simpleAbelianVariety}, there exists an extension $L_1/k$ of degree at most $6^{8g}\cdot(8g)!$ such that there exists a non-torsion point $p_1\in A(L_1)$. The Zariski closure of $\bZ p_1$ in $A_{L_1}$ is an abelian subvariety $A_1$ of dimension $g_1$ by Faltings' theorem. If $g_1=g$, then $p_1$ is a nondegenerate point. If $g_1<g$, then the quotient $A_1'\colonequals A_{L_1}/A_1$ is an abelian variety of dimension $g-g_1\geq 1$ and there is an isogeny $A_{L_1}\sim_{\,L_1}A_1\times A_1'$. A nondegenerate point can then be found using the inductive process below:

Let $i\geq 1$. Suppose that there exist abelian varieties $A_i$ and $A_i'$ over an extension $L_i/k$, where $A_i$ contains a nondegenerate point $p_i\in A_i(L_i)$ and $g_i\colonequals\dim(A_i)<g$, such that there is an isogeny
\begin{equation}
\label{eqn:isogenyAi}
    A_{L_i}
    \,\sim_{\,L_i}\,
    A_i\times A_i'.
\end{equation}
By Lemma~\ref{lemma:dimensionGrowth}, there exists an extension $L_{i+1}/L_i$ with
\[
    [L_{i+1}:L_i]\leq 6^{8(g-g_i)}\cdot(8(g-g_i))!
\]
and a point $p_{i+1}\in (A_i\times A_i')(L_{i+1})$ such that the Zariski closure of $\bZ p_{i+1}$ is an abelian subvariety
$
    A_{i+1}\subset (A_i\times A_i')_{L_{i+1}}
$
of dimension $g_{i+1}>g_i$. If $g_{i+1} = g$, then $p_{i+1}$ induces a nondegenerate point on $A_{L_{i+1}}$ via the isogeny \eqref{eqn:isogenyAi}.
If $g_{i+1}<g$, we define
$$
    A_{i+1}'\colonequals (A_i\times A_i')_{L_{i+1}}/A_{i+1}
$$
and repeat the same argument in this paragraph with $i$ replaced by $i+1$.

As $g$ is fixed, the above process ends up in an abelian variety $A_n$ defined over $L_n/L$, a nondegenerate point $p_n\in A_n$, and an isogeny $A_{L_n}\sim_{\,L_n} A_n$ such that $p_n$ induces a nondegenerate point on $A_{L_n}$. Let $L_0\colonequals k$ and $g_0\colonequals0$. Using the inequalities $g\geq g_{i+1} > g_i$ for $0\leq i\leq n-1$, we estimate the degree of extension by
\begin{gather*}
    [L_n:k]
    = \prod_{i=0}^{n-1}[L_{i+1}:L_{i}]
    \leq\prod_{i=0}^{n-1} 6^{8(g-g_i)}\cdot(8(g-g_i))!\\
    \leq\prod_{i=0}^g 6^{8(g-i)}\cdot(8(g-i))!
    = 6^{4g(g+1)}\prod_{r=0}^{g} (8r)!
    = 6^{4g(g+1)}\prod_{r=1}^{g} (8r)!
\end{gather*}
This finishes the proof.
\end{proof}

\begin{cor}
\label{cor:UPD_abelianVariety}
The collection of abelian varieties of a fixed dimension over fields of characteristic zero satisfies uniform potential density.
\end{cor}

\begin{proof}
By definition, the orbit of a nondegenerate point is Zariski dense, so the statement follows immediately from Theorem~\ref{thm:nondeg_abelianVariety}.
\end{proof}

\subsection{Algebraic groups and their torsors}
\label{subsect:algebraicGroup_Torsor}

Unlike torsors of abelian varieties (see Proposition~\ref{prop:noUPD_genus-one-curve}), we will show that uniform potential density holds for torsors of linear algebraic groups of a fixed dimension over infinite perfect fields. Based on this and the uniform potential density for abelian varieties, we will prove that uniform potential density holds for algebraic groups of a fixed dimension over fields of characteristic zero. In this section, $k$ will denote an infinite perfect field.

\subsubsection{Preliminaries on semisimple groups}
\label{subsubsect:semisimple}

Let us start by gathering necessary background materials and setting up some notations. Our main reference is \cite{Mil17}. For a connected linear algebraic group $G$ over a field $k$, its \emph{radical} $R(G)$ is the maximal connected solvable normal subgroup variety. The group $G$ is \emph{semisimple} if $R(G_{\kbar})$ is trivial, and is further called \emph{almost simple} if $G$ is noncommutative and every proper normal subgroup is finite. We say that $G$ is \emph{geometrically almost simple} if it is almost simple and remains so over $\kbar$.

\begin{thm}[\cite{Mil17}*{Theorem~21.51}]
\label{thm:almost-direct-product}
A semisimple algebraic group $G$ has only finitely many almost simple normal subgroup varieties $G_1,\dots,G_m$ and the morphism
\[\xymatrix@R=0pt{
    G_1\times\cdots\times G_m\ar[r] & G\\
    (g_1,\dots,g_m)\ar@{|->}[r] & g_1\cdots g_m
}\]
is surjective with finite kernel. Moreover, each connected normal algebraic subgroup $H\subset G$ is a product of those $G_i$'s that $H$ contains, and is centralized by those $G_i$'s not inside $H$.
\end{thm}

A semisimple group $G$ over $k$ is \emph{split} if there exists a maximal torus $T\subset G$ isomorphic to a product of copies of $\G_m$ over $k$. By \cite{Tit92}*{Proposition~A1}, a geometrically almost simple group $G$ becomes split over an extension $L/k$ of degree dividing an integer $\delta(G)$ depending only on the type of $G$. Let $v(n)$ denote the $2$-adic valuation of an integer $n$ and assume that $G$ is not of type $E_8$. Then $\delta(G)$ is given explicitly as
$$
\begin{array}{c|ccccccccc}
    & A_n & B_n & C_n & D_{4} & D_{n\geq5} & G_2 & F_4 & E_6 & E_7\\
    \hline
    \delta(G) & 2(n+1) & 2^n & 2^{v(n)+1} & 3\cdot2^6 & 2^{n+v(n)} & 2 & 6 & 2^2\cdot 3^4 & 2^5\cdot 3
\end{array}
$$
For $G$ of type $E_8$, Totaro \cite{Tot04}*{Theorem~0.1} improved the result by Tits to
$$
    \delta(G) = 2^6\cdot 3^2\cdot 5.
$$
In particular, there exists an increasing function on positive integers
\begin{equation}
\label{eqn:increasing_delta}
    \overline{\delta}\colon\bN\longrightarrow\bN
    \quad\text{such that}\quad
    \delta(G)\leq\overline{\delta}(\dim(G))
\end{equation}
for every geometrically almost simple group $G$.

An \emph{inner $k$-form} of $G$ is an algebraic group $G'$ over $k$ together with an isomorphism $f'\colon G_{\kbar}\to G'_{\kbar}$ such that, for every $\sigma\in\Gal(\kbar/k)$, the automorphism
$$
    a_\sigma\colonequals f'^{-1}\circ(\sigma f')
    \colon G\longrightarrow G,
    \quad\text{where}\quad
    \sigma f'\colonequals\sigma\circ f'\circ\sigma^{-1}
$$
is inner. Two inner $k$-forms $(G',f')$ and $(G'',f'')$ are isomorphic if there exists an isomorphism $\varphi\colon G'\to G''$ over $k$ such that the following diagram commutes: 
$$\xymatrix{
    G_{\kbar}\ar[d]_-{f'}\ar[dr]^-{f''} & \\
    G'_{\kbar}\ar[r]_-{\varphi_{\kbar}} & G''_{\kbar}.
}$$
Let $G^{\rm ad}\colonequals G/Z(G)$ where $Z(G)$ is the center of $G$. When $G$ is semisimple and split, the isomorphism classes of inner $k$-forms of $G$ are classified by $H^1(k,G^\mathrm{ad})$ \cite{Mil17}*{\S24.7}. Moreover, the split ones belong to the distinguished class $[G^\mathrm{ad}]$ \cite{Mil17}*{Corollary~23.54}. Note that, if $G$ is geometrically almost simple, then its inner forms are geometrically almost simple as well.

Let $G$ be semisimple and split as before. According to \cite{CGP15}*{Corollary~A.4.11}, there exists a universal covering
$$
    \beta\colon\widetilde{G}\longrightarrow G
$$
where $\widetilde{G}$ is a simply connected semisimple group and $\beta$ is a \emph{central isogeny}, that is, a finite surjective morphism with kernel contained in $Z(\widetilde{G})$. This induces a commutative diagram with exact rows
$$\xymatrix{
    1\ar[r] & Z(\widetilde{G})\ar[d]^-{\alpha}\ar[r] & \widetilde{G}\ar[d]^-{\beta}\ar[r] & \widetilde{G}^\mathrm{ad}\ar[d]^-{\gamma}\ar[r] & 1\\
    1\ar[r] & Z(G)\ar[r] & G\ar[r] & G^\mathrm{ad}\ar[r] & 1.
}$$
The map $\gamma$ is an isomorphism by \cite{BT72}*{Proposition~2.26}. This fact, together with the surjectivity of $\beta$, implies that $\alpha$ is surjective via the snake lemma. Furthermore, we have $\ker(\alpha)\cong\ker(\beta)\cong\pi_1(G)$ since $\beta$ is central. Therefore,
\begin{equation}
\label{eqn:center_quotient}
    Z(G)\cong Z(\widetilde{G})/\pi_1(G).
\end{equation}

Now assume that $G$ is almost simple and split. Then $\widetilde{G}$ is almost simple, and $Z(\widetilde{G})$ is determined by its type as shown in the table \cite{Mil17}*{\S24-c}:
\begin{equation}
\label{eqn:center_universal}
\renewcommand{\arraystretch}{1.3}
\begin{array}{c|cccccccc}
    & A_n & B_n & C_n & D_{2m} & D_{2m+1} & E_6 & E_7 & E_8, F_4, G_2\\
    \hline
    Z(\widetilde{G}) & \mu_{n+1} & \mu_2 & \mu_2 & \mu_2\times\mu_2 & \mu_4 & \mu_3 & \mu_2 & e
\end{array}
\end{equation}
where $\mu_n$ is the $n$-th root of unity and $e$ is the identity element. As a consequence, $Z(G)$ is the quotient of one of the above groups, so its size is bounded by a constant depending only on the type of $G$. We conclude from \eqref{eqn:center_quotient} and \eqref{eqn:center_universal} that, for every split almost simple group $G$, there is an inequality
$$
    |Z(G)|\leq\left\{
    \begin{array}{cl}
        n+1 & \text{if }G\text{ is of type }A_n \\
        4 & \text{otherwise}.
    \end{array}
    \right.
$$
In particular, there exists an increasing function on positive integers
\begin{equation}
\label{eqn:increasing_zeta}
    \overline{\zeta}\colon\bN\longrightarrow\bN
    \quad\text{such that}\quad
    |Z(G)|\leq\overline{\zeta}(\dim(G))
\end{equation}
for every split almost simple group $G$.

\subsubsection{Torsors under connected linear algebraic groups}
\label{subsubsect:torsor_linearAlgebraicGroup}

We are ready to estimate the degrees of field extensions that make a torsor under a connected linear algebraic group trivial. Let us proceed in a step-by-step manner starting from the case of geometrically almost simple groups. We will need the functions $\overline{\delta}$ and $\overline{\zeta}$ defined in \eqref{eqn:increasing_delta} and \eqref{eqn:increasing_zeta}, respectively.

\begin{lemma}
\label{lemma:splitTorsor_almostsimple}
Let $X$ be a $G$-torsor of dimension $r$ where $G$ is a geometrically almost simple algebraic group over a perfect field $k$. Then there exists an extension $L/k$ of degree bounded by $\overline{\zeta}(r)\overline{\delta}(r)^2$ such that $X_L\cong G_L$.
\end{lemma}

\begin{proof}
The exact sequence $1\longrightarrow Z(G)\longrightarrow G\longrightarrow G^\mathrm{ad}\longrightarrow 1$, when considered over any extension $K/k$, induces an exact sequence of pointed sets
$$\xymatrix{
    H^1(K,Z(G_K))\ar[r] &
    H^1(K,G_K)\ar[r] &
    H^1(K,G_K^\mathrm{ad}).
}$$
Let $x\in H^1(k,G)$ be the class of $X$ and $y\in H^1(k,G^\mathrm{ad})$ be its image under the second map above. By passing to an extension $L/k$ of degree bounded by $\overline{\delta}(r)$, we can assume that $G_L$ is split so that $y$ corresponds to the class of an inner $L$-form $G'$ of $G_L$. As $G'$ is also geometrically almost simple of dimension $r$, there exists an extension $L'/L$ of degree bounded by $\overline{\delta}(r)$ such that $G'_{L'}$ is split. This implies that
$$
    y = 0\in H^1(L',G_{L'}^\mathrm{ad})
    \quad\text{and thus}\quad
    x\in H^1(L',Z(G_{L'})).
$$
Let us take an extension $L''/L'$ with $[L'':L']\leq|Z(G_{\kbar})|\leq\overline{\zeta}(r)$ such that $\Gal(\overline{L''}/L'')$ acts trivially on $Z(G_{\kbar})$ (thus trivially on $Z(G_{L''})$). It follows that
$
    H^1(L'',Z(G_{L''})) = 0.
$
In particular, $x = 0$ over $L''$ and we get a desired isomorphism $X_{L''}\cong G_{L''}$. Note that
$$
    [L'':k] = [L'':L'][L':L][L:k]
    \leq\overline{\zeta}(r)\overline{\delta}(r)^2
$$
so the proof is done.
\end{proof}

\begin{lemma}
\label{lemma:geometricallyAlmostSimpleFactors}
Let $G$ be a semisimple algebraic group over a perfect field $k$ of dimension $r>0$. Then there exists an extension $L/k$ of degree bounded by $r$ such that $G_L$ contains a normal subgroup $N$ which is geometrically almost simple with $\dim(N)>0$.
\end{lemma}

\begin{proof}
The algebraic group $G_{\kbar}$ is still semisimple, so it contains almost simple normal subgroups $G_1,\dots,G_m$ as described in Theorem~\ref{thm:almost-direct-product}, and they are geometrically almost simple since we are over $\kbar$. Because $r = \dim(G) > 0$, there exists $G_i$ where $1\leq i\leq m$ such that $\dim(G_i)>0$. Under the action of $\Gal(\kbar/k)$, the orbit of $G_i$ is a subset of $\{G_1,\dots,G_m\}$ and thus has cardinality at most $m\leq r$. It follows that there exists an extension $L/k$ of degree bounded by $r$ such that $N\colonequals G_i$ is defined over $L$.
\end{proof}

\begin{lemma}
\label{lemma:splitTorsor_semisimple}
Let $X$ be a $G$-torsor of dimension $r$ where $G$ is a semisimple algebraic group over a perfect field $k$. Then there exists an extension $L/k$ of degree bounded by a constant $d_r$ which depends only on and increases with $r$ such that $X_L\cong G_L$.
\end{lemma}

\begin{proof}
Let us prove the statement by induction on $r$. Note that the initial case $r=0$ is trivial. Assume that $r\geq 1$. By Lemma~\ref{lemma:geometricallyAlmostSimpleFactors}, after passing to an extension $L/k$ of degree bounded by $r$, there exists a normal subgroup $N\subset G_L$ that is geometrically almost simple with $\dim(N)>0$. This induces an exact sequence of pointed sets
$$\xymatrix{
    H^1(L,N)\ar[r]
    & H^1(L,G_L)\ar[r]
    & H^1(L,G_L/N).
}$$

The torsor $X$ represents a class $[X]\in H^1(L,G_L)$. Let $[Y]\in H^1(L,G_L/N)$ be the image of $[X]$, where $Y$ is a $G_L/N$-torsor. Note that
\begin{itemize}
    \item $G_L/N$ is semisimple by \cite{Mil17}*{Corollary~21.52}, and 
    \item  $\dim(G_L/N)<r$ since $\dim(N)>0$.
\end{itemize}
By applying the induction hypothesis on the $G_L/N$-torsor $Y$, we obtain an extension $L'/L$ of degree bounded by $d_{r-1}$ such that $[Y_{L'}] = 0\in H^1(L',G_{L'}/N_{L'})$. This implies that $[X_{L'}]$ is the image of $[Z]\in H^1(L',N_{L'})$ for some $N_{L'}$-torsor $Z$.

Let $s\colonequals\dim(N)$. Since $N_{L'}$ is geometrically almost simple, Lemma~\ref{lemma:splitTorsor_almostsimple} implies that there exists an extension $L''/L'$ of degree bounded by $\overline{\zeta}(s)\overline{\delta}(s)^2$ such that
$$
    [Z_{L''}] = 0
    \in H^1(L'', N_{L''})
    \quad\text{and thus}\quad
    [X_{L''}] = 0\in H^1(L'', G_{L''}).
$$
This gives a desired isomorphism $X_{L''}\cong G_{L''}$. Note that
$$
    [L'':k]
    = [L'':L'][L':L][L:k]
    \leq\overline{\zeta}(s)\overline{\delta}(s)^2\cdot d_{r-1}\cdot r
    \leq\overline{\zeta}(r)\overline{\delta}(r)^2\cdot d_{r-1}\cdot r
$$
where the last inequality uses $s\leq r$ the fact that $\overline{\delta}$ and $\overline{\zeta}$ are increasing functions. As the last term above depends only on and increases with $r$, the proof is done.
\end{proof}

Before entering the general case about torsors under connected linear algebraic groups, we need to discuss the splitting fields of tori. An algebraic group $T$ over a field $k$ is a \emph{torus} if $T_{\kbar}\cong\G_m^r$. The integer $r$ is called the \emph{rank} of $T$. By definition, $T$ is split if and only if $T\cong\G_m^r$ over $k$.

\begin{lemma}
\label{lemma:splitTorus}
Let $T$ be a torus of rank $r$ over a perfect field $k$. Then there exists an extension $L/k$ of degree bounded by the constant ${\bf c}_r$ defined in Definition~\ref{defn:c_n} such that $T_L$ is split. In particular, we have $H^1(L,T_L)=\{0\}$.
\end{lemma}

\begin{proof}
The action of $\Gal(\kbar/k)$ on the group of characters
$$
    X^*(T)\colonequals
    \Hom(T_{\kbar},\G_m)
    \cong\bZ^r
$$
induces a homomorphism $\Gal(\kbar/k)\longrightarrow\GL(r,\bZ)$, whose kernel is isomorphic to $\Gal(\kbar/L)$ for some finite Galois extension $L/k$. As $\Gal(\kbar/L)$ acts on $X^*(T)$ trivially, this implies that $T_L$ is split, that is, $T_L\cong\G_m^r$, due to \cite{Mil17}*{Theorem~12.23}. Since $\Gal(L/k)$ appears as a finite subgroup of $\GL(r,\bZ)$, we have $[L:k] = |\Gal(L/k)|\leq\textbf{c}_r$. The last statement follows from the fact that $H^1(L,\G_m^r) = \{0\}$ (see, e.g., \cite{Mil17}*{Corollary~3.47}).
\end{proof}

\begin{prop}
\label{prop:splitTorsor_linearAlgGp}
Let $X$ be a $G$-torsor of dimension $r$ where $G$ is a connected linear algebraic group over a perfect field $k$. Then there exists an extension $L/k$ of degree bounded by the multiplication $d_r\cdot{\bf c}_r$ such that $X_L\cong G_L$. Here $d_r$ and ${\bf c}_r$ are increasing functions in $r$ given in Lemma~\ref{lemma:splitTorsor_semisimple} and  Definition~\ref{defn:c_n}, respectively.
\end{prop}

\begin{proof}
Let $U\subset G$ be the unipotent radical. By \cite{San81}*{Lemma~1.13}, there is a canonical isomorphism $H^{1}(k,G)\cong H^{1}(k,G/U)$, which reduces the proof to the case when $G$ is reductive.

For a reductive group $G$, its radical $R\subset G$ is a torus \cite{Mil17}*{Corollary~17.62}. Let $s$ be the rank of $R$. By Lemma~\ref{lemma:splitTorus}, there exists an extension $K/k$ of degree bounded by ${\bf c}_s$ such that 
$
    H^1(K,R_K)=\{0\}.
$
This induces an injection
$$\xymatrix{
    0\ar[r] & H^1(K,G_K)\ar[r] & H^1(K,G_K/R_K).
}$$
Consider $[X_K]\in H^1(K,G_K)$ as an element in $H^1(K,G_K/R_K)$. Let $t\colonequals\dim(G/R)$ and notice that $G_K/R_K$ is semisimple. By Lemma~\ref{lemma:splitTorsor_semisimple}, there exists an extension $L/K$ of degree bounded by a constant $d_t$ such that $[X_L]$ is trivial in $H^1(L,G_L/R_L)$ and thus is trivial in $H^1(L,G_L)$. It follows that $X_L\cong G_L$. To complete the proof, observe that
$$
    [L:k] = [L:K][K:k]
    \leq d_t\cdot {\bf c}_s
    \leq d_r\cdot {\bf c}_r
$$
where the last inequality follows from the inequalities $s<r$ and $t<r$ as well as the fact that $d_r$ and ${\bf c}_r$ increase with $r$.
\end{proof}

\begin{cor}
\label{cor:UPD_torsor_linearAlgGp}
Retain the condition of Proposition~\ref{prop:splitTorsor_linearAlgGp} and assume further that $k$ is infinite. Then there exists an extension $L/k$ of degree bounded by $d_r\cdot{\bf c}_r$ such that $X(L)$ is dense. In particular, uniform potential density holds for the collection of torsors of connected linear algebraic groups of a fixed dimension over infinite perfect fields.
\end{cor}

\begin{proof}
By Proposition~\ref{prop:splitTorsor_linearAlgGp}, there exists an extension $L/k$ of degree bounded by $d_r\cdot{\bf c}_r$ such that $X_L\cong G_L$. Since $k$ is perfect, $G$ is unirational over $k$ \cite{Bor91}*{Theorem~18.2~(ii)}, which implies that $G(k)$ is dense as $k$ is infinite. This proves that $X(L)$ is dense via the isomorphism $X_L\cong G_L$.
\end{proof}

\begin{cor}
\label{cor:quasi-split}
Retain the condition of Proposition~\ref{prop:splitTorsor_linearAlgGp} and assume further that $G$ is quasi-split, that is, contains a Borel subgroup over $k$. Let $r$ denote the rank of $G$, that is, the rank of a maximal torus of $G$. Then there exists an extension $L/k$ of degree bounded by $\textbf{c}_r$ such that $X_L\cong G_L$.
\end{cor}

\begin{proof}
According to \cite{Ste65}*{Theorem~11.1}, there exists a torus $T\subset G$ such that the class $[X]\in H^1(k,G)$ lies in the image of $H^1(k,T)$. Note that $t\colonequals\operatorname{rank}(T)\leq r$. By Lemma~\ref{lemma:splitTorus}, there exists an extension $L/k$ of degree bounded by $\textbf{c}_t\leq\textbf{c}_r$ such that $H^1(L,T_L)=\{0\}$, which implies the statement.
\end{proof}

\begin{rmk}
\label{rmk:lowerBound}
The question about minimal degrees of the field extensions that split torsors under connected linear algebraic groups is also intensively studied. For relevant results and further references in this direction, we refer the reader to \cites{RY01,Tot04}.
\end{rmk}

\subsubsection{Connected algebraic groups}
\label{subsubsect:UPD_AlgebraicGroup}

Now we have almost everything we need to proof the uniform potential density for connected algebraic groups in characteristic zero.

\begin{lemma}
\label{lemma:density_shortExactSequence}
Let $A$ be an abelian variety over a field $k$ of characteristic zero with an extension
$$\xymatrix{
    1\ar[r] & N\ar[r] & G\ar[r]^-{\pi} & A\ar[r] & 1
}$$
by a linear algebraic group $N$ over $k$. Suppose that $A(k)$ contains a nondegenerate point $p$ such that the coset $N_1\colonequals\pi^{-1}(p)\subset G$ contains a dense subset of $k$-rational points. Then $G(k)$ is dense in $G$.
\end{lemma}

\begin{proof}
For every $m\in\bZ$, we define $N_m\colonequals\pi^{-1}(p^m)\subset G$. Note that $N_0 = N$ and that every $N_m$ is an $N$-torsor under the action
$$
    N\times N_{m}\to N_{m}
    : (x,y)\mapsto xy.
$$
By hypothesis, $N_1$ is a trivial $N$-torsor. This implies that every $N_m$ is trivial since one can produce a $k$-rational point on $N_m$ via the morphisms
\begin{itemize}
    \item $N_1^m\to N_m:x\mapsto x^m$ if $m>0$,
    \item $N_{-m}\to N_m:x\mapsto x^{-1}$ if $m<0$.
\end{itemize}
It follows that $N_m(k)$ is dense in $N_m$ for all $m\in\bZ$. As a result, we obtain a Zariski dense subset
\[
    \bigcup_{m\in\bZ} N_m(k)\subset G(k)
\]
and hence $G(k)$ is dense in $G$.
\end{proof}

\begin{thm}
\label{thm:UPD_algebraicGroup}
The collection of connected algebraic groups of a fixed dimension over fields of characteristic zero satisfies uniform potential density.
\end{thm}

\begin{proof}
Let $G$ be a connected algebraic group over a field $k$ of characteristic zero. By the Barsotti--Chevalley--Rosenlicht theorem (\cites{Bar55,Che60,Ros56}; see also \cite{Con02}*{Theorem~1.1}), there is a short exact sequence of algebraic $k$-groups
$$\xymatrix{
    1\ar[r] & N\ar[r] & G\ar[r]^-{\pi} & A\ar[r] & 1
}$$
where $N$ is a linear closed subgroup of $G$ and $A$ is an abelian variety.

Let $g\colonequals\dim(A)$ and $r\colonequals\dim(G)$. Notice that $g\leq r$. By Theorem~\ref{thm:nondeg_abelianVariety}, there exists an extension $L/k$ of degree bounded by
$$
    6^{4g(g+1)}\prod_{i=1}^{g}(8i)!
    \leq 6^{4r(r+1)}\prod_{i=1}^{r}(8i)!
$$
such that $A(L)$ contains a nondegenerate point. On the other hand, $N$ is unirational over $k$ \cite{Bor91}*{Theorem~18.2~(ii)}, which implies that $N(k)$ is dense in $N$ as $k$ is infinite in our case. Note that the coset $N_1\colonequals\pi^{-1}(p)$ is an $N$-torsor under the action
$$
    N\times N_1\to N_1
    : (x,y)\mapsto xy.
$$
Let $n\colonequals\dim(N)$. By Corollary~\ref{cor:UPD_torsor_linearAlgGp}, there exists an extension $L'/L$ of degree bounded by $d_n\cdot{\bf c}_n$ such that $N_1(L')$ is dense in $N_1(L')$. (See Lemma~\ref{lemma:splitTorsor_semisimple} and  Definition~\ref{defn:c_n} for the definition of $d_r$ and ${\bf c}_r$.)

The fact that $N_1(L')$ is dense in $N_1(L')$, together with the existence of a nondegenerate point in $A(L)$ (and thus in $A(L')$), implies that $G(L')$ is dense in $G$ by Lemma~\ref{lemma:density_shortExactSequence}. To complete the proof, notice that $d_n\cdot{\bf c}_n\leq d_r\cdot{\bf c}_r$ since $d_r$ and ${\bf c}_r$ are increasing in $r$. Therefore,
$$
    [L':k] = [L':L][L:k]
    \leq d_r\cdot{\bf c}_r
    \cdot 6^{4r(r+1)}\prod_{i=1}^{r}(8i)!
$$
where the last term depends only on $r=\dim(G)$.
\end{proof}

\section{Uniform potential density for elliptic K3 surfaces}
\label{sect:UPD_ellipticK3}

In this section, we prove that uniform potential density holds for the collection of elliptic K3 surfaces over number fields. Throughout \S\ref{subsect:NS_EllipticK3}, the ground field $k$ is assumed to be perfect, or arbitrary with $\operatorname{char}(k)>5$. In the remaining part of the section, we will specialize to number fields unless otherwise noted.

\subsection{Picard groups and automorphisms}
\label{subsect:NS_EllipticK3}

Let $X$ be an elliptic K3 surface over $k$ and let $k^s$ denote the separable closure of $k$. We want to find finite extensions $L/k$ where
\begin{itemize}
    \item all the elements in $\Pic(X_{k^s})$ descend to $\Pic(X_L)$, or
    \item a chosen element in $\Aut(X_{k^s})$ is defined.
\end{itemize}
In each case, we will show that the minimal degree of such extensions $L/k$ is bounded by a constant independent of $X$. These results will be used to prove the uniform potential density respectively in the cases $\rho(X_{k^s})\leq 19$ and $\rho(X_{k^s})=20$.

\subsubsection{Defining fields of Picard groups}
\label{subsubsect:defineNeronSeveri}

Let $X$ be an elliptic K3 surface over a field $k$. For every separable extension $L/k$, there are natural injections
\begin{equation}
\label{eqn:twoInclusionsNS}
\xymatrix{
    \Pic(X_L)\ar@{^(->}[r]
    & \Pic(X_{k^s})^{\Gal(k^s/L)}\ar@{^(->}[r]
    & \Pic(X_{k^s}).
}
\end{equation}
Our goal is to find a separable extension $L/k$ such that the two injections become isomorphisms with the degree $[L:k]$ bounded. Let us discuss the injection on the left first.

In general, for every proper and geometrically integral variety $X$ over a field $k$, there is an exact sequence \cite{Poo17}*{Corollary~6.7.8}:
$$\xymatrix{
    0\ar[r]
    & \Pic(X)\ar[r]
    & \Pic(X_{k^s})^{\Gal(k^s/k)}\ar[r]
    & \Br(k)\ar[r]
    & \Br(X)
}$$
If $X(k)\neq\emptyset$, then the morphism $\Spec(k)\to X$ induces a section $\Br(X)\to\Br(k)$. Under this condition, the last map of the sequence is injective, which forces the second map to be an isomorphism
$$\xymatrix{
    \Pic(X)\ar[r]^-\sim
    & \Pic(X_{k^s})^{\Gal(k^s/k)}.
}$$
Therefore, to make the first injection in \eqref{eqn:twoInclusionsNS} an isomorphism, it is sufficient to find a separable extension $L/k$ such that $X(L)\neq\emptyset$.

Let us analyze this property for a general elliptic surface. Recall that the Euler--Poincar\'e characteristic $\chi(Z)$ is defined by
$$
    \chi(Z) = \sum_{i} (-1)^i\dim H_{\mathrm{et}}^i(Z_{\kbar},\bQ_\ell),
    \quad\text{where }\;
    \ell\neq\operatorname{char}(k).
$$
For an elliptic surface $\pi\colon X\to C$, we have \cite{CD89}*{Proposition~5.1.6}:
\begin{equation}
\label{eqn:euler_ellipticSurf}
    \chi(X) = \sum_{b\in C(\overline{k})}\chi(X_b).
\end{equation}
Here the sum is finite as $\chi(X_b) = 0$ if $X_b$ is smooth. Given a fiber $X_b$ over $b\in C(k^s)$, we denote by $\#X_b$ the number of its geometrically irreducible components (ignoring the multiplicities). For a singular $X_b$, we have $\chi(X_b) = \#X_b$ or $\chi(X_b) = \#X_b+1$ depending on its type. The assumption that $k$ is perfect or arbitrary with $\operatorname{char}{k}>5$ implies that the singular fibers are all of classical Kodaira types. (See \cite{CD89}*{Proposition~3.1.1} for the perfect case and \cite{LLR04}*{Appendix~A} for the other case.)

\begin{lemma}
\label{lemma:ELS_ellipticSurface}
Let $\pi\colon X\to C$ be an elliptic surface over $k$ and let $X_b\subset X$ be a singular fiber over $b\in C(k^s)$ with minimal possible $\chi(X_b)$. Then, for any reduced and geometrically irreducible component $D\subset X_b$, there exists a separable extension $L/k$ of degree at most $2\chi(X)$ such that $D$ is birational to $\bP^1_L$. In particular, we have $X(L)\supset D(L)\neq\emptyset$.
\end{lemma}

\begin{proof}
Consider the set
$$
    \Sigma^s\colonequals\{
        b\in C(k^s)
        \mid
        X_b\text{ is singular}
    \}\neq\emptyset.
$$
Then there exists a separable extension $L_1/k$ of degree at most $|\Sigma^s|$ such that the point $b\in \Sigma^s$, and thus $X_b$, is defined over $L_1$. According to Kodaira's table, the component $D$ appears in $X_b$ either as a reduced component or with higher multiplicity divisible by only $2$, $3$, or $5$. Because $\operatorname{char}(k)>5$, this implies that $D$ is defined over a separable extension $L_2/L_1$ of degree at most $\#X_b\leq\chi(X_b)$. The normalization of $D$ is a smooth conic, which is isomorphic to $\bP^1_L$ over a separable extension $L/L_2$ of degree at most $2$. In particular, $D_L$ is birational to $\bP^1_L$ and thus $X(L)\supset D(L)\neq\emptyset$. About the degree of $L/k$, we have
$$
    [L:k]
    = [L:L_2][L_2:L_1][L_1:k]
    \leq 2\chi(X_b)|\Sigma^s|.
$$
Using the minimality of $\chi(X_b)$ and \eqref{eqn:euler_ellipticSurf}, we obtain
$$
    2\chi(X_b)|\Sigma^s|
    \leq 2\sum_{b\in\Sigma^s}\chi(X_b)
    \leq 2\chi(X).
$$
Combine this inequality with the one above, we conclude that $[L:k]\leq 2\chi(X)$.
\end{proof}

\begin{lemma}
\label{lemma:sepSingularFiber}
Every elliptic K3 surface $X$ over $k$ has a singular fiber over $k^s$.
\end{lemma}

\begin{proof}
In the case that $k$ is perfect, we have $k^s = \overline{k}$. Suppose, to the contrary, that $X_{\overline{k}}$ has no singular fiber. Then $\chi(X) = 0$ by \eqref{eqn:euler_ellipticSurf}, which contradicts to the fact that $\chi(X) = 24$ for every K3 surface $X$. Therefore, $X$ has a singular fiber over $k^s=\overline{k}$.

Now assume that $k$ is arbitrary with $\operatorname{char}{(k)}>5$. When $X$ has a section, it admits a Weierstrass model $W\to\bP^1$ (see, e.g., \cite{Huy16}*{\S11.2}):
$$
    y^2z = 4x^3 - g_2xz^2 - g_3z^3
$$
where $g_2\in H^0(\bP^1,\cO_{\bP^1}(8))$ and $g_3\in H^0(\bP^1,\cO_{\bP^1}(12))$. Its discriminant is given by
$$
    \Delta = g_2^3 -27g_3^2
    \;\in\; H^0(\bP^1,\cO_{\bP^1}(24)).
$$
If $W$ has no singular fiber over $k^s$, then $\operatorname{char}(k)$ divides $\deg(\Delta) = 24$, which contradicts to the hypothesis that $\operatorname{char}(k)>5$. Therefore, $W$ has a singular fiber over $k^s$. This implies that $X$ has a singular fiber over $k^s$ as well.

In general, we can associate $X$ with its Jacobian fibration $J(X)$, which is an elliptic K3 surface with a section. There exists a dominant rational map $f\colon X\dashrightarrow J(X)$ preserving the fibrations, which acts by mapping a smooth fiber to its Jacobian. (See \S\ref{subsubsect:JacobEllipticK3} for a brief introduction on $J(X)$.) We have proved that $J(X)$ has a singular fiber over $k^s$, so $X$ has a singular fiber over $k^s$ as $f$ maps a singular fiber to a singular fiber.
\end{proof}

\begin{cor}
\label{cor:ELS_ellipticK3}
Let $X$ be an elliptic K3 surface over $k$. Then there exists a separable extension $L/k$ of degree at most $48$ such that, for every separable extension $L'/L$, the natural inclusion
\[\xymatrix{
    \Pic(X_{L'})\ar@{^(->}[r] &
    \Pic(X_{k^s})^{\Gal(k^s/L')}
}\]
is an isomorphism
\end{cor}

\begin{proof}
The surface $X$ has a singular fiber over $k^s$ by Lemma~\ref{lemma:sepSingularFiber}, so Lemma~\ref{lemma:ELS_ellipticSurface} guarantees the existence of a separable extension $L/k$ of degree at most $2\chi(X) = 2\cdot 24 = 48$ such that $X(L)\neq\emptyset$. Hence $X(L')\neq\emptyset$ for every separable extension $L'/L$, which implies the conclusion.
\end{proof}

Recall that by Theorem~\ref{thm:minkowski}, the group $\GL(n,\bZ)$ contains, up to isomorphism, only finitely many finite subgroups and $\mathbf{c}_n$ denotes the largest possible cardinality of such subgroup. Now let us study the injection on the right in \eqref{eqn:twoInclusionsNS}.

General versions of the following lemma for K3 surfaces have appeared in the literature under the assumption of the existence of a rational point (e.g. \cite{Huy16}*{Lemma~17.2.6}). The version we present here is specific to elliptic K3 surfaces without this assumption.

\begin{lemma}
\label{lemma:defineNeronSeveri}
Let $X$ be an elliptic K3 surface over $k$ and denote $\rho\colonequals\rho(X_{k^s})$. Then there exists a separable extension $L/k$ of degree at most $48\mathbf{c}_\rho$ such that there is a natural isomorphism
\[\xymatrix{
    \Pic(X_L)\ar[r]^-\sim & \Pic(X_{k^s}).
}\]
\end{lemma}

\begin{proof}
By Corollary~\ref{cor:ELS_ellipticK3}, there exists a separable extension $L_1/k$ of degree at most $48$ such that, for every separable extension $L'/L_1$, there is a natural isomorphism
\begin{equation}
\label{eqn:NS-GaloisInvNS}
\xymatrix{
    \Pic(X_{L'})\ar[r]^-\sim &
    \Pic(X_{k^s})^{\Gal(k^s/L')}.
}
\end{equation}
The Galois action on $\Pic(X_{k^s})$ induces a group homomorphism
\[\xymatrix{
    \sigma\colon\Gal(k^s/L_1)\ar[r] & \GL(\rho,\bZ).
}\]
Since $\Pic(X_{k^s})$ is finitely generated, the kernel is finite index in $\Gal(k^s/k)$, corresponding to a finite extension $L/L_1$. Hence, $\Gal(k^s/L)=\ker(\sigma)$. This identifies the finite group $\Gal(L/L_1)$ as a subgroup of $\GL(\rho,\bZ)$, and so $[L:L_1] = |\Gal(L/L_1)|$ is bounded by $\mathbf{c}_\rho$. We conclude that there exists a Galois extension $L/L_1$ of degree at most $\mathbf{c}_\rho$ such that
\begin{equation}
\label{eqn:galoisInvNS-NSbar}
    \Pic(X_{k^s})^{\Gal(k^s/L)}
    =\Pic(X_{k^s}).
\end{equation}
The desired isomorphism follows immediately from \eqref{eqn:NS-GaloisInvNS} and \eqref{eqn:galoisInvNS-NSbar}.
Moreover, we have $[L:k] = [L:L_1]\cdot[L_1:k] \leq \mathbf{c}_\rho\cdot 48$.
\end{proof}

\subsubsection{Defining fields of automorphisms on K3 surfaces}
\label{subsubsect:defineAutK3}

Let us introduce some notations first. For a K3 surface $X$ over $k$, the group $\Aut(X_{k^s})$ acts on $\Pic(X_{k^s})$ by pulling back line bundles,
\[\xymatrix@R=.2em{
    \Aut(X_{k^s})\ar[r] & \mathrm{O}(\Pic(X_{k^s}))
    :f\ar@{|->}[r] & f^*.\\
}\]
We will also consider the action of $\Gal(k^s/k)$ on $\Aut(X_{k^s})$ by
$$\xymatrix@R=.1em{
    \Gal(k^s/k)\times\Aut(X_{k^s})\ar[r] & \Aut(X_{k^s})\\
    (\sigma, f)\ar@{|->}[r] & f^\sigma\colonequals\sigma^{-1}f\sigma.
}$$
For every separable extension $L/k$, we have
$$
    \Aut(X_L)=\Aut(X_{k^s})^{\Gal(k^s/L)}
$$
under the above action.

\begin{prop}
\label{prop:automorphism}
Let $X$ be a K3 surface over $k$ that satisfies $\Pic(X) = \Pic(X_{k^s})$. Then, for every $f\in\Aut(X_{k^s})$, there exists an integer $n>0$ such that $f^n\in\Aut(X)$.
\end{prop}

\begin{proof}
By hypothesis, $\Gal(k^s/k)$ acts on $\Pic(X_{k^s})$ trivially. Let us fix some  $f\in\Aut(X_{k^s})$ and $\sigma\in\Gal(k^s/k)$. Then, for every $\mathcal{L}\in\Pic(X_{k^s})$, we have
$$
    (f^\sigma)^*\mathcal{L}
    = (\sigma^{-1}f\sigma)^*\mathcal{L}
    = \sigma f^*(\sigma^{-1})\mathcal{L}
    = \sigma(f^*\mathcal{L})
    = f^*\mathcal{L}
$$
where the hypothesis on the Galois action is applied to the last two equalities.
Hence
$$
    (f^\sigma f^{-1})^*\mathcal{L}
    = (f^{-1})^*(f^\sigma)^*\mathcal{L}
    = (f^*)^{-1}(f^\sigma)^*\mathcal{L}
    = \mathcal{L}.
$$
This implies that $f^\sigma f^{-1}\in\ker(\Aut(X_{k^s})\to\mathrm{O}(\Pic(X_{k^s})))$.

Applying the above argument to $f^\ell$ where $\ell\in\bZ$, we conclude that
$$
    \{(f^\sigma)^\ell f^{-\ell} \mid \ell\in\bZ\}
    \subset\ker(\Aut(X_{k^s})\to\mathrm{O}(\Pic(X_{k^s}))).
$$
The kernel on the right is finite by \cite{Huy16}*{Proposition~5.3.3}, so the set on the left is finite as well. In particular, there exist distinct integers $i$ and $j$ such that
$$
    (f^\sigma)^i f^{-i}
    = (f^\sigma)^j f^{-j},
    \text{ or equivalently, }
    f^{i-j}
    = (f^\sigma)^{i-j}
    = (f^{i-j})^\sigma.
$$
In other words, $f^m$, where $m\colonequals|i-j|>0$, is fixed by $\sigma\in\Gal(k^s/k)$.

The integer $m$ may depend on $\sigma$. However, as $f$ is defined over some finite separable extension over $L$, it is fixed by some finite index subgroup $H\subset\Gal(k^s/k)$. Let $\sigma_1,\dots,\sigma_r\in\Gal(k^s/k)$ be representatives for the cosets. Then we can associate each $\sigma_i$ an integer $m_i$ as above so that $f^{m_i}$ is fixed by $\sigma_i$. Define $n\colonequals\prod_{i=1}^rm_i$. Then $f^n$ is fixed by each $\sigma_i$ and thus is fixed by $\Gal(k^s/k)$, so $f^n\in\Aut(X)$. This finishes the proof.
\end{proof}

\subsection{Jacobian fibrations and rational multisections}
\label{subsect:jacobian_RatMultisection}

A main ingredient in Bogomolov and Tschinkel's \cite{BT00} proof of potential density is about the existence of a non-torsion rational multisection. In this section, we show that one can find such a multisection over a finite extension of bounded degree.

\subsubsection{Jacobian fibrations of elliptic K3 surfaces}
\label{subsubsect:JacobEllipticK3}

Let us start with a brief review on the Jacobian fibrations of elliptic K3 surfaces. For the general foundations on torsors and their classifying spaces, we refer the reader to \cite{Poo17}*{Chapter~6} and \cite{Mil80}*{Chapter~III}. We also refer the reader to \cite{Huy16}*{\S11.4 and \S11.5} for more details of the following materials treated over algebraically closed fields.

Let $\pi\colon X\to\bP^1$ be an elliptic K3 surface over a perfect field $k$ and let $\eta$ denote the generic point of the base $\bP^1$. The generic fiber $X_\eta$ is a genus one curve over the function field $K\colonequals k(\bP^1)$, so we can talk about its Jacobian 
$$
    J(X_\eta)\colonequals\Pic^0_{X_\eta/K}.
$$
The Jacobian $J(X_\eta)$ has a natural structure of an elliptic curve as it contains the $K$-rational point $[\cO_{X_\eta}]$. On the other hand, it is also a surface over $k$ and thus admits a minimal model $J(X)$. The surface $J(X)$, called the \emph{Jacobian fibration} of $X$, is an elliptic K3 surface equipped with a section from the point $[\cO_{X_\eta}]$. In the following, an elliptic K3 surface with a section will be called a Jacobian fibration.

Given a Jacobian fibration $\pi\colon J\to\bP^1$, its \'etale local sections form a sheaf $\cJ$ of abelian groups on $\bP^1$. There is a natural bijection
\[
    H^1(\bP^1, \cJ)
    \longleftrightarrow\{
        \text{Elliptic K3 surfaces }X\to\bP^1\text{ with }J(X)\cong J
    \}/\sim
\]
where $X\sim X'$ if and only if they are isomorphic as elliptic fibrations. It is well known that the group $H^1(\bP^1, \cJ)$ is torsion. Hence, to every elliptic K3 surface $X$ with $J(X)\cong J$, we can assign a positive integer $\operatorname{ind}(X)$ called the \emph{index} of $X$, defined as the order of its class $[X]$ in $H^1(\bP^1, \cJ)$. Equivalently, the index of $X$ is the index of the image of the group homomorphism
\[\xymatrix@R=0pt{
    \Pic(X)\ar[r] & \bZ\\
    D\ar@{|->}[r] & D\cdot X_b
}\]
where $X_b$ is the fiber over a general $b\in\bP^1(\kbar)$. 

Given $[X],[Y]\in H^1(\bP^1,\cJ)$ such that $[X]=n[Y]$ for some positive integer $n$, there exists a commutative diagram of dominant $k$-rational maps of elliptic fibrations:
\begin{equation}
\label{eqn:existRatMap}
\vcenter{\xymatrix{
    Y \ar@{-->}[r] \ar@{-->}[dr] & X \ar@{-->}[d] \\
    & J.
}}
\end{equation}
Let $\ell$ and $m$ be the indices of $[X]$ and $[Y]$, respectively. Then $\ell = nm$, and over a point $b\in\bP^1(k)$ where the fibers are smooth, these maps act as
\[\xymatrix{
    Y_b\cong\Pic^1_{Y_b/k} \ar@{-->}[r] \ar@{-->}[dr] & X_b\cong\Pic^n_{Y_b/k} \ar@{-->}[d] \\
    & J_b\cong\Pic^\ell_{Y_b/k}.
}\]
For example, the downward arrow on the right maps $\mathcal{L}$ to $\mathcal{L}^{\otimes m}$.

The following lemma is well known to experts. We include a proof from the MathOverflow post \cite{MathOverflow} by Lior Bary-Soroker based on the \emph{embedding problems} (c.f. \cite{FJ08}*{Definition~16.4.1}).

\begin{lemma}
\label{lemma:hilbertianGalois}
Let $k$ be a Hilbertian field (e.g., the function field of $\bP^1$ over a number field) and $M$ a nonzero finite $\Gal(\overline{k}/k)$-module.
Then, $H^1(k,M)$ is infinite.
\end{lemma}

\begin{proof}
Let $G_k\colonequals\Gal(\overline{k}/k)$ and $U\subset G_k$ be an open subgroup such that the action on $M$ factors through $G:=G_k/U$. For each integer $n>0$, we consider the semidirect product of groups $$
    \Gamma_n\colonequals M^n\rtimes G
$$
where $G$ acts on $M^n$ by $g\cdot (m_1,\ldots,m_n)=(gm_1,\ldots,gm_n)$. Since the field $k$ is Hilbertian, there exists a surjective homomorphism $\psi\colon G_k\to \Gamma_n$ and a commutative diagram
$$\xymatrix{
    G_k \ar[dr]_-r\ar[r]^-\psi & \Gamma_n \ar[d]^-{p} \\
    & G
}$$
where $r\colon G_k\to G$ is the quotient map and $p\colon\Gamma_n\to G$ is the second projection (c.f. \cite{FJ08}*{Proposition~16.4.5}).

For each $i=1,\dots,n$, we define
$$
    \pi_i\colon \Gamma_n\longrightarrow \Gamma_1
    : ((m_1,\ldots,m_n),g)\mapsto (m_i,g)
$$
and let $\psi_i=\pi_i\circ \psi$. Since $M$ is nontrivial, it contains at least two distinct elements $m$ and $m'$. For any $i\neq j$, let $\gamma_{ij}\in M^n$ be an element whose $i$-th entry is $m$ and $j$-th entry is $m'$. By the surjectivity of $\psi$, there exists $g\in G_k$ such that $\psi(g)=(\gamma_{ij},0)\in \Gamma_n$. Since $r(g)=p\circ \psi(g)=0$, it follows that $g\in U$. Notice that $\psi_i(g) = (m,0)\neq (m',0) = \psi_j(g)$. This shows that, for any $i\neq j$, there exists $u\in U$ such that $\psi_i(u)\neq \psi_j(u)$.

For each $1\leq i\leq n$, define $c_i\colon G_K\to M$ to be the composition
$$\xymatrix{
    c_i\colon G_k \ar[r]^-{\psi_i} & \Gamma_1 = M\rtimes G\ar[r]^-{q} & M.
}$$
where $q$ is the projection to $M$. We claim that each $c_i$ is a $1$-cocycle, that is, $c_i(gh)=c_i(g)+gc_i(h)$ for any $g,h\in G_k$. For the sake of simplicity, we denote the images of $g$ and $h$ under $r\colon G_k\to G$ as $\overline{g}$ and $\overline{h}$, respectively. Since $\psi_i$ is a homomorphism,
$$
    q(\psi_i(g)\psi_i(h))
    = q((c_i(g),\overline{g})(c_i(h),\overline{h}))
    = q(c_i(g)+\overline{g}c_i(h),\overline{gh})
    = c_i(g)+\overline{g}c_i(h)
$$
which verifies the claim. Next we show that $c_i$ and $c_j$ are distinct modulo $1$-coboundaries for any $i\neq j$. Recall that there exists $u\in U$ such that $\psi_i(u)\neq\psi_j(u)$. This implies that $c_i(u)\neq c_j(u)$ as $\psi_\ell(u) = (c_\ell(u),0)$ for all $\ell$. Since $d(u) = 0$ every $1$-coboundary $d\colon G_k\to M$ and $u\in U$, the claim follows. We have shown that each $c_i$ gives a distinct class in $H^1(k,M)$, which shows this group is infinite by letting $n\to\infty$.
\end{proof}

\begin{lemma}
\label{lemma:infiniteTorsion}
Let $\pi\colon J\to\bP^1$ be a Jacobian fibration over a number field $k$ and let $\cJ$ be the associated sheaf on $\bP^1$. Then $H^1(\bP^1,\cJ)[n]$ contains infinitely many elements for any $n>1$.
\end{lemma}

\begin{proof}
Let $\eta$ denote the generic point of $\bP^1$. Let $K/\kbar(\eta)$ be a finite Galois extension such that $J_\eta(K)[n]=J_\eta(\overline{K})[n]\isom (\bZ/n\bZ)^2$ is the full $n$-torsion subgroup. Let $C_{\kbar}\to\bP^1_{\kbar}$ be the Galois cover of smooth projective curves over $\kbar$ corresponding to $K/\kbar(\eta)$. Denote by $G$ the Galois group $\Gal(K/k(\eta))$. The Hochschild--Serre spectral sequence
$$
    E^{p,q}_2=H^p(G, H^q(C_{\kbar},\cJ))\implies H^{p+q}(\bP^1,\cJ)
$$
and the isomorphism $H^0(C_{\kbar},\cJ)\isom J_{\eta}(K)$ give an injection
$
    H^1(G,J_{\eta}(K))\injects H^1(\bP^1,\cJ),
$
which restricts to an injection
\begin{equation}
\label{eqn:injectTorsions}
    H^1(G,J_{\eta}(K))[n]\injects H^1(\bP^1,\cJ)[n].
\end{equation}

Consider $M:=J_\eta(K)$ as a $G$-module. The Kummer sequence
$$\xymatrix{
    0\ar[r]
    & M[n]\ar[r]
    & M\ar[r]^-{\times n}
    & M\ar[r]
    & 0
}$$
induces the exact sequence
$$
    0\longrightarrow
    \frac{J_{\eta}(k(\eta))}{nJ_{\eta}(k(\eta))}\longrightarrow
    H^1(G,M[n])\longrightarrow
    H^1(G,M)[n]\longrightarrow
    0.
$$
The quotient $J_{\eta}(k(\eta))/nJ_{\eta}(k(\eta))$ is finite by the Mordell--Weil theorem over function fields. By Lemma~\ref{lemma:hilbertianGalois}, the group $H^1(k(\eta),M[n])$ is infinite. Hence, by enlarging $K$, we can make $H^1(G,M[n])$ arbitrarily large. (Note that $M[n]$ does not change upon enlarging $K$.)
It follows then that $H^1(G,M)[n]$ can also be made arbitrarily large, and thus $H^1(\bP^1,\cJ)[n]$ must be infinite by (\ref{eqn:injectTorsions}).
\end{proof}

\begin{lemma}
\label{lemma:groupTheory}
Let $G$ be an abelian group that contains an element of prime order $p$. Then, for every $x\in G$ whose order is finite and not divisible by $p$,
there exists $z\in G$ such that $pz = x$ and $p\mid\ord(z)$.
\end{lemma}

\begin{proof}
Let $d\colonequals\ord(x)$.
First we show that there exists $\beta\in\bZ$ such that
\begin{equation}
\label{eqn:betaModp}
    \beta d\equiv 1\;(\text{mod }p)
    \quad\text{and}\quad
    \beta d\not\equiv 1\;(\text{mod }p^2).
\end{equation}
There exists $\beta$ satisfying the first equation since $p$ and $d$ are coprime.
If $\beta d\not\equiv 1\;(\text{mod }p^2)$, then we are done.
If $\beta d\equiv 1\;(\text{mod }p^2)$, then
\[
    (\beta+p)d
    \equiv 1+pd
    \not\equiv 1\;(\text{mod }p^2),
\]
where the second relation uses the assumption $p\nmid d$.
Note that $(\beta+p)d\equiv 1\;(\text{mod }p)$.
Therefore, we can replace $\beta$ by $\beta+p$ to fulfill \eqref{eqn:betaModp}.

Let $\alpha\in\bZ$ be such that $\alpha p+\beta d=1$. Note that by \eqref{eqn:betaModp}, we have $p\nmid\alpha$. Pick an element $y\in G$ of order $p$ and define $z\colonequals \alpha(x+y)$. Then
$$
    pz = p\alpha (x+y) = p\alpha x = (1-\beta d)x = x - \beta dx = x.
$$
It follows that $\ord(z)\mid pd$.
We claim that $\ord(z)\nmid d$. Indeed, if $\ord(z)$ divides $d$, then
\[
    0 = dz = d\alpha(x+y) = d\alpha y,
\]
which implies that $p\mid d\alpha$ and thus $p\mid\alpha$, contradiction. Therefore, we have $p\mid\ord(z)$.
\end{proof}

\begin{prop}
\label{prop:divideClass}
Let $\pi\colon J\to\bP^1$ be a Jacobian fibration over a number field $k$. For every $[X]\in H^1(\bP^1,\cJ)$ and every prime $p\nmid\operatorname{ind}(X)$, there exists $[Y]\in H^1(\bP^1,\cJ)$ such that $[X]=p[Y]$ and $p\mid\operatorname{ind}(Y)$.
\end{prop}
\begin{proof}
By Lemma~\ref{lemma:infiniteTorsion}, the group $H^1(\bP^1,\cJ)[p']$ is infinite for every prime $p'$. In particular, for every prime $p$ as in the statement, $H^1(\bP^1,\cJ)$ contains an element of prime order $p$. Now apply Lemma~\ref{lemma:groupTheory} with $G = H^1(\bP^1,\cJ)$ to get the desired $Y$.
\end{proof}

\subsubsection{Existence of non-torsion rational multisections}
\label{subsubsect:nontorsionRatmultisection}

Let $\pi\colon X\to\bP^1$ be an elliptic K3 surface over a perfect field $k$ and let $\cJ$ be the associated Jacobian sheaf. Following \cite{BT00}*{Definition~3.7}, we say that a multisection $\cM\subset X$ is \emph{torsion of order $m$} or \emph{m-torsion} if $m$ is the smallest positive integer such that for every $b\in\bP^1(\kbar)$ and $p,q\in\cM\cap X_b(\kbar)$, the class $[p]-[q]$ lies in $\cJ_b[m]$.

We call $\cM$ \emph{non-torsion} or an \emph{nt-multisection} if there exists $b\in\bP^1(\kbar)$ such that $[p]-[q]\in\cJ_b$ is torsion-free for some $p,q\in\cM\cap X_b(\kbar)$. Recall that there exists such a $b\in\bP^1(\kbar)$ if and only if the statement holds for a general point in $\bP^1(\kbar)$ by \cite{BT00}*{Lemma~3.8}.

\begin{lemma}
\label{lemma:torsion-torsion}
Pick any $[X],[Y]\in H^1(\bP^1,\cJ)$ that satisfy $[X] = n[Y]$ and let $\varphi\colon Y\dashrightarrow X$ be the map given in \eqref{eqn:existRatMap}. Then a multisection $\cM\subset Y$ is non-torsion if and only if its image $\varphi(\cM)\subset X$ is non-torsion.
\end{lemma}

\begin{proof}
It is easy to see that $\cM$ is torsion implies that $\varphi(\cM)$ is torsion. Let us assume that $\cM$ is non-torsion. Then there exist $b\in\bP^1(\kbar)$ and $p,q\in Y_b(\kbar)$ such that $Y_b$ is smooth and $[p]-[q]\in\Pic^0(Y_b)$ is non-torsion.
Choose $p$ and $\varphi(p)$ as origins for $Y_b$ and $X_b$, respectively. We can identify $\varphi|_{Y_b}$ as the isogeny $[n]\colon\cJ_b\to\cJ_b$, where $[n]$ is the multiplication by $n$.
Under this identification, $q$ is non-torsion, and thus $\varphi(q)$ is non-torsion. It follows that $[\varphi(q)]-[\varphi(p)]\in\Pic^0(X_b)$ is non-torsion, so $\varphi(\cM)\subset X$ is non-torsion.
\end{proof}

Given a relatively minimal Jacobian fibration $J\to\bP^1$, we consider the invariant $c(J)$ defined by Hassett \cite{Has03}*{Corollary~9.5}. Suppose that, with respect to Kodaira's table of singular fibers, $J$ has $n_0$ fibers of type~${\rm I}_0^*$, $n_1'$ fibers of type~${\rm I}_a$ with $a>0$, $n_1''$ fibers of type~${\rm I}_a^*$ with $a>0$, $n_2$ fibers of type~${\rm II}$ or ${\rm II}^*$, $n_3$ fibers of type~${\rm III}$ or ${\rm III}^*$, and $n_4$ fibers of types~${\rm IV}$ or ${\rm IV}^*$. Then $c(J)$ is defined by
\begin{equation}
\label{eqn:express_c(J)}
    c(J)
    = \frac{1}{2}n_0 + n_1' + n_1''
    + \frac{5}{6}n_2 + \frac{3}{4}n_3 +\frac{2}{3}n_4
    - 2.
\end{equation}

\begin{lemma}
\label{lemma:genusMultisection}
Let $X$ be an elliptic K3 surface over a number field $k$ with associated Jacobian fibration $J$ and Jacobian sheaf $\cJ$. Suppose that $c(J) > 0$. Then there exists a prime $p$ and a class $[Y]\in H^1(\bP^1,\cJ)$ such that $[X] = p[Y]$ and every rational multisection on $Y$ is non-torsion.
\end{lemma}

\begin{proof}
By Proposition~\ref{prop:divideClass}, for every prime $p\nmid\operatorname{ind}(X)$, there exists $[Y]\in H^1(\bP^1,\cJ)$ such that $[X]=p[Y]$ and $\operatorname{ind}(Y) = pd$ for some $d > 0$. Let $[Y'] = d[Y]\in H^1(\bP^1,\cJ)$. Then $\operatorname{ind}(Y')=p$ and we have rational maps between elliptic fibrations:
\[\xymatrix{
    Y\ar@{-->}[r]^-{\varphi} & Y'\ar@{-->}[r] & J.
}\]
Let $\cM\subset Y$ be a rational multisection and suppose, to the contrary, that it is torsion. Then the multisection $\varphi(\cM)\subset Y'$ is torsion by Lemma~\ref{lemma:torsion-torsion}. Corollary~3.10 of \cite{BT00} then implies that $\varphi(\cM)$ admits a surjection onto one of the $p$-torsion multisections either on $Y'$ or $J-\{\text{zero section}\}$. However, the irreducible components of these $p$-torsion multisections have positive genus provided that $p\gg 0$ due to $c(J)>0$ and \cite{Has03}*{Corollaries~8.3, 9.5, and Theorem~9.9}. Hence $\varphi(\cM)$, and thus $\cM$, is non-torsion if $p\gg 0$.
\end{proof}

\begin{lemma}
\label{lemma:existenceRatMult}
Let $X$ be an elliptic K3 surface over a number field $k$. Then there exists a finite extension $L/k$ of degree at most $48^2\mathbf{c}_{20}$ such that there is a rational multisection $\cM\subset X_L$ defined over $L$.
\end{lemma}

\begin{proof}
By hypothesis, the group $\Pic(X)$ contains primitive classes $H$ and $E$ where $H$ is ample and $E$ defines an elliptic fibration. They span a sublattice $\langle H, E\rangle\subset\Pic(X)$ with intersection products
$$
\begin{pmatrix}
2a & r\\
r & 0
\end{pmatrix},
\;\text{ where }\;
a>0\;\text{ and }\; r>0.
$$
Define $L\colonequals rH-aE$. Then $L^2 = 0$ and $L\cdot H =ar>0$, which imply that $L$ is effective by Riemann--Roch. We claim that there exist an elliptic curve $E'$ and $(-2)$-curves $C_1,\dots,C_\ell$ over the algebraic closure $\overline{k}$ such that
$$
    L = m_0E' + \sum_{i=1}^\ell m_i C_i
    \;\in\;\Pic(X_{\kbar}),
    \;\text{ where }\;
    m_0,\dots,m_\ell>0.
$$
Indeed, if $L$ is nef, then $L = m_0E'$ for some elliptic curve $E'$ and positive integer $m_0$ by \cite{Huy16}*{Proposition~2.3.10}, so the claim holds in this case. If $L$ is not nef, one can follow the procedure in \cite{Huy16}*{Remark~8.2.13} to find $(-2)$-curves $C_i$ and positive integers $m_i$, where $1\leq i\leq \ell$, such that
$L-\sum_{i=1}^\ell m_iC_i$ is effective and nef. Using \cite{Huy16}*{Proposition~2.3.10} again, we conclude that $L-\sum_{i=1}^\ell m_iC_i = m_0E'$ for some elliptic curve $E'$ and $m_0>0$, which proves the claim.

Let us denote the fibration defined by $E$ as $\pi_E\colon X\to\bP^1$ and the fibration defined by $E'$ as $\pi_{E'}\colon X\to\bP^1$.
We claim that $\pi_E$ admits a rational multisection given by
\begin{enumerate}[label=(\alph*)]
    \item\label{multisection_(-2)-curves}
    one of the $(-2)$-curves $C_i$, where $1\leq i\leq\ell$, or
    \item\label{multisection_singularFibers}
    a reduced and geometrically irreducible component $D$ of a singular fiber $X_b$ with minimal possible $\chi(X_b)$ in the fibration $\pi_{E'}\colon X\to\bP^1$.
\end{enumerate}
In the case that $E = E'$, we have
\[
    \sum_{i=1}^\ell m_i(C_i\cdot E)
    = L\cdot E
    = (rH-aE)\cdot E
    = r^2 > 0.
\]
Hence $\pi_E$ admits a multisection as in \ref{multisection_(-2)-curves}. Suppose that $E\neq E'$
and that $C_i$ is not a multisection of $\pi_E$ for every $i$. Then every fiber of $\pi_{E'}$ is as a multisection of $\pi_E$ since
$$
    E'\cdot E =
    \frac{1}{m_0}\left(
        L - \sum_{i=1}^{\ell}m_iC_i
    \right)\cdot E
    = \frac{1}{m_0}L\cdot E
    = \frac{1}{m_0}(rH-aE)\cdot E
    = \frac{1}{m_0}\cdot r^2> 0.
$$
In particular, we can find a multisection as in \ref{multisection_singularFibers} by picking a singular fiber $X_b$ with minimal possible $\chi(X_b)$ and then choosing a reduced and geometrically irreducible component $D\subset X_b$ that appears as a multisection of $\pi_E$.

Now let us find the field extensions where the multisections of types \ref{multisection_(-2)-curves} and \ref{multisection_singularFibers} are defined. By Lemma~\ref{lemma:defineNeronSeveri}, there exists an extension $L_1/k$ of degree at most $48\mathbf{c}_\rho \leq 48\mathbf{c}_{20}$ such that $\Pic(X_{L_1})\cong\Pic(X_{\kbar})$. Note that $C_i$ for $i=1,\dots,\ell$ are defined over $L_1$ as they are $(-2)$-curves. Therefore, the proof is done if we are in case~\ref{multisection_(-2)-curves}.
Suppose that we are in case~\ref{multisection_singularFibers}. Note that the class $E'$, and thus the fibration $\pi_{E'}$, is defined over $L_1$. By Lemma~\ref{lemma:ELS_ellipticSurface}, there exists an extension $L/L_1$ of degree at most $2\chi(X) = 48$ over which $D$ is birational to $\bP^1_L$.
Moreover,
$
    [L:k] = [L:L_1][L_1:k] = 48^2\mathbf{c}_{20},
$
so the proof is done.
\end{proof}

\begin{prop}
\label{prop:ntRatMultisection}
Let $X$ be an elliptic K3 surface over a number field $k$. Suppose that its associated Jacobian fibration $J$ satisfies $c(J) > 0$. Then there exists an extension $L/k$ of degree at most $48^2\mathbf{c}_{20}$ and a non-torsion rational multisection $\cM\subset X_L$ over $L$.
\end{prop}

\begin{proof}
By Lemma~\ref{lemma:genusMultisection}, there exists $[Y]\in H^1(\bP^1,\cJ)$ such that $[X]=p[Y]$ for some prime $p$ and every rational multisection on $Y$ is non-torsion. This fact, together with Lemma~\ref{lemma:existenceRatMult}, implies the existence of an extension $L/k$ of degree at most $48^2\mathbf{c}_{20}$ and a non-torsion rational multisection $\cM\subset Y_L$. The image of $\cM$ under the map $Y_L\dashrightarrow X_L$ gives a rational multisection on $X_L$ that, by Lemma~\ref{lemma:torsion-torsion}, is non-torsion as well.
\end{proof}

\subsection{Proof of uniform potential density}
\label{subsect:proof_UPD_ellipticK3}

Let us finish the proof of uniform potential density for elliptic K3 surfaces.

\begin{thm}
\label{thm:ellk3density}
For every elliptic K3 surface $\pi\colon X\to\bP^1$ over a number field $k$, there exists an extension $L/k$ of degree at most $4608\mathbf{c}_{20}$ such that $X(L)$ is dense in Zariski topology.
\end{thm}

The Jacobian fibration $J = J(X)$ determines an invariant $c(J)$ that was defined by Hassett \cite{Has03}*{Corollary~9.5}. (See \eqref{eqn:express_c(J)} for the definition.) We proceed the proof with respect to $c(J)$ and the geometric Picard number $\rho(X_{\kbar})$ in the following three cases:
\begin{enumerate}[label=(\alph*)]
    \item\label{c(J)-positive}
    $c(J) > 0$, which allows us to use Proposition~\ref{prop:ntRatMultisection}.
    \item\label{c(J)-leq0_Kummer}
    $c(J)\leq 0$ and $\rho(X_{\kbar}) < 20$. In this case, $X$ is a Kummer surface \cite{Has03}*{Remark~9.7}.
    \item\label{c(J)-leq0_singular}
    $\rho(X_{\kbar}) = 20$, that is, $X$ is \emph{singular} in the context of Shioda and Inose \cite{SI77}.
\end{enumerate}

\begin{proof}[Proof of Theorem~\ref{thm:ellk3density} for case~\ref{c(J)-positive}]

By Proposition~\ref{prop:ntRatMultisection}, there exists an extension $L/k$ of degree at most $48^2\mathbf{c}_{20}$ and a non-torsion rational multisection $\cM\subset X_L$ defined over $L$. The normalization of $\cM$ is a smooth rational curve, hence, upon replacing $L$ by a quadratic extension of itself if necessary, we can assume that $\cM(L)$ is dense in $\cM$. Now we have
\[
    [L:k]\leq 2\cdot 48^2\mathbf{c}_{20}
    = 4608\mathbf{c}_{20}.
\]
Note that $\cM$ defines a rational map $\varphi\colon X\dashrightarrow J\colonequals J(X)$ by sending a point $x\in X_{\kbar}$ lying on a smooth fiber to
$$
    \varphi(x)\colonequals d_{\cM}[x] - [\cM_{\pi(x)}]\in\Pic^0(X_{\pi(x)})
$$
where $d_{\cM}$ is the degree of $\cM$.

By Merel's theorem (see, e.g. \cite{Sil09}*{VIII, Theorem~7.5.1}), there exists a constant $N$ depending only on $[L:\bQ]$ such that every elliptic curve $E$ over $L$ satisfies
$
    |E_\text{tor}(L)|\leq N.
$
For every positive integer $n$, the torsion points of order $n$ on the fibers of $J_{\kbar}$ form a multisection $\mathcal{T}_n\subset J$. Let us define
\[
    \mathcal{T}\colonequals
    \bigcup_{n\leq N}\mathcal{T}_n\subset J.
\]
Since $\cM$ is non-torsion, there exist $x,y\in\cM(L)$ lying on the same smooth fiber $X_{\pi(x)}$ such that $[x]-[y]\in\Pic^0(X_{\pi(x)})$ is non-torsion. It follows that
\[
    \varphi(x)-\varphi(y)
    = d_{\cM}([x] - [y])
    \in J_{\pi(x)}
\]
is also non-torsion. This implies that $\varphi(\cM)$ is not contained in $\mathcal{T}$, so the intersection $\varphi(\cM)\cap T$ is zero-dimensional and thus contains only finitely many closed points. Hence, for all but finitely many $z\in\cM(L)$, the image $\varphi(z)\in\varphi(\cM(L))$ is either of infinite order or torsion of order greater than $N$. The latter case is impossible due to Merel's theorem. Thus for all but finitely many $z\in\cM(L)$, the fiber $J_{\pi(z)}$ has infinitely many $L$-points. Translating points in $\cM(L)$ by points in $J_{\pi(z)}$ gives a dense subset of $L$-points on $X_L$.
\end{proof}

\begin{rmk}
The strategy in the second part of the proof of case~\ref{c(J)-positive} has been used previously in literature to prove the density of rational points on surfaces that admits multiple elliptic fibrations over number fields. See, for example, \cite{vL12}*{\S3}, where a more general result was proved.
\end{rmk}

\begin{proof}[Proof of Theorem~\ref{thm:ellk3density} for case~\ref{c(J)-leq0_Kummer}]

By \cite{Has03}*{Proposition~9.6}, the fibration $X\to\bP^1$ is isotrivial with four degenerate fibers of type $I_0^*$, represented by the diagram
$$
I_0^*:\quad
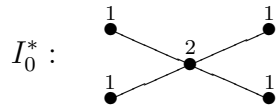
\begin{aligned}
    \xymatrix@R=7pt{
    *=0{\bullet}\ar@{-}[dr]^(-.07){1}^(.93){2}
    && *=0{\bullet}\ar@{-}[dl]_(-.07){1} \\
    & *=0{\bullet} & \\
    *=0{\bullet}\ar@{-}[ur]^(.07){1}
    && *=0{\bullet}\ar@{-}[ul]_(.07){1}
}\end{aligned}
$$
The locus $B\subset\bP^1$ under the degenerate fibers determines a double cover $E\to\bP^1$ branched at $B$. The fiber product $X\times_{\bP^1}E$ is non-normal along the preimages of the non-reduced components of the degenerate fibers. Its normalization, denoted as $\widetilde{Y}$, contains sixteen $(-1)$-curves that are the preimages of the reduced components of the degenerate fibers. By contracting these $(-1)$-curves, we obtain an isotrivial elliptic fibration $f\colon Y\to E$ with smooth fibers. Let us organize these into a diagram:
$$
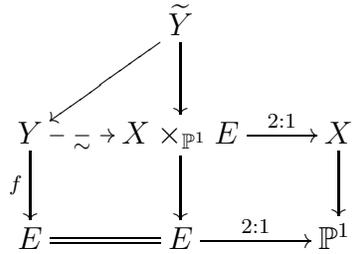
\xymatrix{
    & \widetilde{Y}\ar[dl]\ar[d] & \\
    Y\ar[d]_-f\ar@{-->}[r]_-\sim & X\times_{\bP^1}E\ar[d]\ar[r]^-{2:1} & X\ar[d] \\
    E\ar@{=}[r] & E\ar[r]^-{2:1} & \bP^1.
}$$

We claim that $Y$ is a \emph{hyperelliptic surface} in the sense of \cite{BM77}*{\S3}, that is, its Kodaira dimension $\kappa(Y)$ is zero and the Albanese variety ${\rm Alb}(Y)$ is an elliptic curve. First of all, one can verify that $\chi(Y) = 0$ by tracking the construction of $Y$. Together with the fact that $K_Y^2 = 0$ for an elliptic surface in general, we obtain $\chi(\cO_Y) = 0$ by Noether's formula. Plugging this into \cite{BM77}*{Theorem~2~(iv)}, we conclude that the canonical bundle $\omega_Y$ of $Y$ is the pullback of a line bundle $\cLL$ on $E$ of degree zero. Hence
$$
    \dim H^0(Y,\omega_Y^{\otimes m})
    = \dim H^0(E,\cLL^{\otimes m})
    = 1
    \quad\text{for }\; m \gg 0.
$$
This shows that $\kappa(Y) = 0$. On the other hand, the second Betti number of $Y$ equals $b_2(Y) = b_2(E) = 2$ \cite{CD89}*{Corollary~5.2.2}, so the irregularity $q(Y) = b_2(Y)/2 = 1$. Thus $\operatorname{Alb}(Y)$ has dimension one, that is, is an elliptic curve. This proves the claim.

Because $E$ is a double cover of $\bP^1$, there exists an extension $L_1/k$ of degree at most $2$ such that $E(L_1)\neq\emptyset$, or equivalently, $E_{L_1}$ is an elliptic curve. By Corollary~\ref{cor:UPD_ellipticCurve}, there exists an extension $L_2/L_1$ of degree at most $2$ such that $E(L_2)$ contains a non-torsion point $p$. In particular, the set $E(L_2)$ is infinite, which implies that $Y_{L_2}$ is a finite quotient of $E_{L_2}\times E'$ for some elliptic curve $E'$ over $L_2$ by \cite{BM77}*{Theorem~4}. (Although \cite{BM77} works over an algebraically closed field, the same arguments work over $K$ assuming $E(K)$ is infinite.) Applying Corollary~\ref{cor:UPD_ellipticCurve} to $E'$, we get an extension $L/L_2$ of degree at most $2$ such that $E'(L)$ contains a non-torsion point $p'$. Then the point $(p,p')\in E_L\times E'_L$ is nondegenerate as it spans a subset whose projections onto $E_L$ and $E'_L$ are Zariski dense. In particular, the $L$-rational points are dense in $E_L\times E'_L$ and thus dense in $Y_L$, which implies that $X(L)$ is dense as well. About the degree of $L/k$,
$$
    [L:k] = [L:L_2][L_2:L_1][L_1:k]
    \leq 2\cdot 2\cdot 2 = 8
    < 4608\mathbf{c}_{20}.
$$
This completes the proof for case~\ref{c(J)-leq0_Kummer}.
\end{proof}

\begin{proof}[Proof of Theorem~\ref{thm:ellk3density} for case~\ref{c(J)-leq0_singular}]

Assume that $\rho(X_{\kbar})= 20$. By Lemma~\ref{lemma:defineNeronSeveri} and Proposition~\ref{prop:automorphism}, there exists a finite extension $L/k$ of degree at most $48\mathbf{c}_{20}$ such that
\begin{enumerate}[label=(\roman*)]
    \item\label{singular_Neron}
    $\Pic(X_{\kbar})\cong\Pic(X_L)$, and
    \item\label{singular_automorphism}
    every $f\in\Aut(X_{\kbar})$ satisfies $f^n\in\Aut(X_L)$ for some integer $n$.
\end{enumerate}
Note that every elliptic fibration on $X_{\kbar}$ is defined over $L$ due to \ref{singular_Neron}.
According to \cite{SI77}*{\S5}, there exists an elliptic fibration
$$\xymatrix{
    \pi\colon X_L\ar[r] & \bP^1
}$$
which admits infinitely many sections over $\kbar$. In other words, the generic fiber $X_\eta$, as a genus one curve over $L(\bP^1)$, contains infinitely many rational points over $K\colonequals\kbar(\bP^1)$. Consider $X_\eta$ as an elliptic curve over $K$. The Mordell--Weil theorem implies that $X_\eta(K)$ is finitely generated, so there exists $\Sigma\in X_\eta(K)$ that is non-torsion. Translation by $\Sigma$ induces an automorphism $\sigma\in\Aut(X_{\kbar})$ of infinite order. Using \ref{singular_automorphism} above, we conclude that there exists an integer $n$ such that
\[
    \tau\colonequals\sigma^n\in\Aut(X_L)
\]
which is again of infinite order. Note that $\tau$ can also be considered as an automorphism on the genus one curve $X_\eta$.

By Lemma~\ref{lemma:existenceRatMult}, after an extension $L'/L$ of degree at most $48$, there exists a rational multisection $\cM$ over $L'$. Extending by a quadratic extension if necessary, we can assume that $\cM(L')$ is infinite.
Let $\cM'\cong\bP^1_{L'}$ be the normalization of $\cM$ and consider the base change
\[\xymatrix{
    \widetilde{X}\colonequals\cM'\times_{\bP^1}X_{L'}\ar[d]\ar[r] & X_{L'}\ar[d]_{\pi}\\
    \cM'\ar[r] & \bP^1.
}\]
Let $\xi$ denote the generic point of $\cM'$. Then the generic fiber $\widetilde{X}_{\xi}$ can be seen as an elliptic curve since $\widetilde{X}$ admits a section
\[
    \widetilde{\cM}\colonequals
    \cM'\times_{\bP^1}\cM.
\]
Moreover, $\tau$ induces an automorphism $\widetilde{\tau}$ on $\widetilde{X}$ of infinite order that acts on $\widetilde{X}_{\xi}$ as an automorphism of a genus one curve.

The orbit of $\widetilde{\cM}_{\xi}$ as a point on $\widetilde{X}_{\xi}$ under $\widetilde{\tau}$ is infinite. Indeed, if not, then $\widetilde{\tau}^\ell(\widetilde{\cM}_{\xi})=\widetilde{\cM}_{\xi}$ for some integer $\ell$, hence $\tau^\ell$ acts as an automorphism on the elliptic curve $(\widetilde{X}_{\xi},\widetilde{\cM}_{\xi})$. However, the automorphism group of an elliptic curve is finite \cite{Sil09}*{III, Theorem~10.1}, which implies that $\widetilde{\tau}$ is of finite order, contradiction. As a consequence, the orbit of $\widetilde{\cM}_{\xi}$ is infinite. It follows that the orbit of the set $\widetilde{\cM}(L')$ under $\widetilde{\tau}$ is dense in $\widetilde{X}$. This produces a dense subset of $L'$-rational points on $X_{L'}$ via the morphism $\widetilde{X}\to X_{L'}$. Note that
\[
    [L':k] = [L':L][L:k]
    \leq2\cdot48^2\cdot\mathbf{c}_{20}
    = 4608\mathbf{c}_{20}
\]
so we finish the proof.
\end{proof}

\section{Brauer groups and uniform potential density}
\label{sect:BrauerUPD}

In this section, we show that uniform potential density holds for a collection of varieties provided that certain conjectures related to the Brauer--Manin obstruction hold.

\subsection{Brauer--Manin obstructions and related conjectures}
\label{subsect:conjectureBrauerManin}

For a comprehensive introduction to the Brauer--Manin obstruction, we refer the reader to \cite{Poo17}*{Chapter 8}. Let us recall only necessary definitions here. Given a smooth and geometrically integral variety $X$ over a number field $k$, we consider its Brauer group $\Br(X)$ and the following groups
\begin{align*}
    \Br_0(X)&\colonequals\im(\Br(k)\longrightarrow\Br(X))\\
    \Br_1(X)&\colonequals\ker(\Br(X)\longrightarrow\Br(X_{\kbar})).
\end{align*}
Let $\Omega_k$ be the set of places of $k$. Each $v\in\Omega_k$ induces a homomorphism
$$\xymatrix{
    \inv_v\colon\Br(k_v)\ar@{^(->}[r] & \bQ/\bZ
}$$
called a \emph{local invariant}. It is an isomorphism if $v$ is finite, identifies $\Br(k_v)$ with $\frac{1}{2}\bZ/\bZ$ if $v$ is real, and is the zero map if $v$ is complex.
Using this, we can define the \emph{adelic Brauer--Manin paring}
\begin{equation}
\label{eqn:BMpairing}
\begin{aligned}
    \xymatrix@R=0pt{
    \Br(X)\times X(\A_k)\ar[r] & \bQ/\bZ\\
    (\alpha,\{x_v\})\ar@{|->}[r] & \sum_{v\in\Omega_k}\inv_v(\alpha(x_v)).}
\end{aligned}
\end{equation}
This pairing is trivial on $\Br_0(X)$, so it can be considered as a pairing between $\Br(X)/\Br_0(X)$ and $X(\A_k)$ as well. The subset in $X(\A_k)$ orthogonal to all elements in $\Br(X)$ under this pairing is denoted as $X(\A_k)^{\Br}$.

Let $\cC=\{(X,k)\}$ be a collection of smooth and geometrically integral varieties $X$ over number fields $k$. Recall that $\cC$ satisfies
\begin{itemize}
    \item \emph{uniform boundedness of exponent for Brauer groups (UBEB)} if, for every integer $d>0$, there exists a constant $c_d$ such that $c_d\alpha=0$ for all $\alpha\in\Br(X_L)/\Br_0(X_L)$ where $X/k\in\cC$ and $L/k$ is an extension such that $[L:k]\leq d$;
    \item \emph{Brauer--Manin obstruction is the only one to weak approximation (BMWA)} if for all $X/k\in\cC$ and finite extension $L/k$, we have $X(L)$ dense in $X(\A_L)^{\Br}$.
    \item \emph{Uniform potential local solubility (UPLS)} if, there exists a constant $c$ such that for all $X/k\in \cC$, there exists an extension $L/k$ of degree $\leq c$ such that $X_L$ is everywhere locally soluble.
\end{itemize}

For example, the collection $\cC$ of geometrically rational surfaces over number fields satisfies UBEB since this holds for del Pezzo surfaces (see, e.g., \cite{Cor07}*{Theorem~4.1} for a list of possible Brauer groups for del Pezzo surfaces) and conic bundles (see, e.g., \cite{CTS19}*{Lemma~10.2.2 and Proposition~10.2.3}). It is conjectured by Colliot-Th\'{e}l\`{e}ne that if $\cC$ consists of rationally connected variety defined over number fields, then $\cC$ satisfies BMWA \cite{CT03}*{p.174}. Evidently, uniform potential density for $\cC$ implies UPLS.

For K3 surfaces over number fields, Skorobogatov-Zarhin \cite{SZ08}*{Theorem~1.2} proved that $\Br(X)/\Br_0(X)$ is always finite. Drawing analogy from Merel's theorem for elliptic curves, V\'arilly-Alvarado \cite{VA17} further conjectured that a stronger condition than UBEB holds when restricted to families of K3 surfaces based on the lattice structures of their N\'eron--Severi groups. See also \cite{OS18}, \cite{CC20}, and \cite{OSZ19} for some progress toward this conjecture.

\begin{conj}[\cite{VA17}*{Conjecture~4.6}]
\label{conj:UBEB}
Fix an integer $d>0$ and a primitive lattice $\Delta$ in the K3 lattice $\Lambda_{K3}\colonequals U^{\oplus 3}\oplus E_8(-1)^{\oplus 2}$. Let $\cC=\{(X,k)\}$ be the collection of K3 surfaces $X$ over number fields $k$ satisfying $\Pic(X_{\kbar})\isom \Delta$ and $[k:\bQ]\leq d$. Then there exists a constant $c_d$ such that $|\Br(X)/\Br_0(X)|< c_d$ for any $X/k\in\cC$;.
\end{conj}

For BMWA, we have the following conjecture by Skorobogatov.

\begin{conj}[\cite{Sko09}]\label{conj:BMWA}
K3 sufaces over a number fields satisfy BMWA.
\end{conj}

This conjecture was proved in the special case of Kummer surfaces assuming the finiteness of the relevant Tate--Shafarevich groups \cite{SSD05}. It is far from being proved in general, but various computations have been done to support it, see, for example \cite{Bri06}.

\subsection{A sufficient condition for uniform  potential density}
\label{subsect:UPD-UBEB-BMWA}

Our main result in this section relates UBEB, BMWA, and UPLS to uniform potential density. We first show that one can control the growth of Brauer groups using UBEB. Many variants of the following lemma have appeared in the literature; see Remark \ref{rmk:surjectivity_Brauer}.

\begin{lemma}
\label{lemma:sujectivity_Brauer}
Let $X$ be a smooth projective geometrically integral variety over a number field $k$ such that $\Pic(X_{\kbar})$ is torsion free of finite rank, and that UBEB holds for $X$. For any integer $d$, there exists an extension $L/k$ such that for any extension $K/k$ of degree $d$ linearly disjoint from $L$, the natural map
$$
    \Br(X)/\Br_0(X)\longrightarrow \Br(X_K)/\Br_0(X_K)
$$
is an isomorphism.
\end{lemma}

\begin{proof}
Since $P\colonequals\Pic(X_{\kbar})$ is finitely generated, the action by $\Gal(\kbar/k)$ on $P$ factors through a finite quotient, say $\Gal(L/k)$. Since $\Gal(\kbar/L)$ acts trivially on $P$, we have
$
H^1(L,P)=0.
$
By the inflation restriction exact sequence,
\begin{equation}
\label{eqn:infl-rest}
H^1(\Gal(L/k),P)\isom H^1(k,P), \quad H^2(\Gal(L/k),P)\injects H^2(k,P).
\end{equation}

Let $d>0$ be an integer. By the hypothesis on UBEB, there exists a constant $c_d$ such that for every extension $M/k$ of degree at most $d$,
\begin{equation}
\label{eqn:torsioniso}
    \Br(X_M)/\Br_0(X_M)[c_d]=\Br(X_M)/\Br_0(X_M).
\end{equation}
By \cite{CTS19}*{Proposition~4.2.5}, the set $\Br(X_{\kbar})[c_d]$ is finite. Therefore, upon enlarging $L$ by a finite extension if necessary, we can assume that $\Gal(\kbar/L)$ acts trivially on $\Br(X_{\kbar})[c_d].$ Let $K/k$ be a degree $d$ extension linearly disjoint from $L/k$. Then
\begin{equation}
\label{eqn:isom_torsion}
    \Br(X_{\kbar})[c_d]^{\Gal(\kbar/K)}=\Br(X_{\kbar})[c_d]^{\Gal(\kbar/k)}
\end{equation}

The spectral sequence
$$
    E^{p,q}_2
    = H^p(k,H^q(X_{\kbar},\G_m))
    \implies H^{p+q}(X,\G_m)
$$
gives rise to the following exact sequence (see e.g. \cite{CTS13}*{Proposition~1.3~(i)})
\begin{equation}\label{eqn:brexact}
\xymatrix{
    0\ar[r] & H^1(k,P)\ar[r] & \Br(X)/\Br_0(X)\ar[r] & \Br(X_{\kbar})^{\Gal(\kbar/k)}\ar[r] & H^2(k,P).
}
\end{equation}
Enlarging $L$ by a finite extension if necessary, we assume that the image of $\Br(X_{\kbar})[c_d]^{\Gal(\kbar/k)}$ lies in $H^2(\Gal(L/k),P)$. Replacing $H^i(k,P)$ by $H^i(\Gal(L/k),P)$ in \eqref{eqn:brexact} and taking the $c_d$-torsion elements in \eqref{eqn:brexact} gives the following commutative diagram by functoriality,
$$\small\xymatrix{
    H^1(\Gal(L/k),P)[c_d]\ar@{^(->}[r]\ar[d] & \frac{\Br(X)}{\Br_0(X)}[c_d]\ar[r]\ar[d] & \Br(X_{\kbar})^{\Gal(\kbar/k)}[c_d]\ar[r]\ar[d] & H^2(\Gal(L/k),P)\ar[d]\\
    H^1(\Gal(LK/K),P)[c_d]\ar@{^(->}[r] & \frac{\Br(X_K)}{\Br_0(X_K)}[c_d]\ar[r] & \Br(X_{\kbar})^{\Gal(\kbar/K)}[c_d]\ar[r] & H^2(\Gal(LK/K),P).
}$$
Note that the first two rows of are exact because $\frac{\Br(X)}{\Br_0(X)}$ and $ \frac{\Br(X_K)}{\Br_0(X_K)}$ are both $c_d$-torsion groups by assumption. From \eqref{eqn:isom_torsion}, we know that the third vertical map is an isomorphism. The first and last vertical map is an isomorphism since $\Gal(L/k)\isom\Gal(LK/K)$ compatible with the action on $P$. The third vertical map is an isomorphism on the $c_d$-torsion parts by \eqref{eqn:isom_torsion}. 
Then the five lemma implies that the second vertical map is an isomorphism on the $c_d$-torsion parts, from which we deduce it is an isomorphism on the entire group by \eqref{eqn:torsioniso}.
\end{proof}

\begin{rmk}
\label{rmk:surjectivity_Brauer}
Special cases of Lemma~\ref{lemma:sujectivity_Brauer} have appeared in the literature for varieities where UBEB is known. In \cite{Lia13} and \cite{Har97}, the authors assumed the finiteness of the geometric Brauer group $\Br(X_{\kbar})$. Ieronymou \cite{Ier21} proved it for K3 surfaces $X$, in which case it follows from \cite{OS18} that the $X$ satisfies UBEB.
\end{rmk}

\begin{thm}
\label{thm:bm}
Let $\cC$ be a collection of smooth projective geometrically integral varieties $X/k$ over number fields such that $H^1(X,\cO_X)=0$ and $\Pic(X_{\kbar})$ is torsion-free of bounded rank. If $\cC$ satisfies UBEB, BMWA, and UPLS then it also satisfies uniform potential density.
\end{thm}

\begin{proof}
Using UPLS, up to a bounded extension of $k$, we can assume that $X_k$ is everywhere locally soluble. By UBEB, there exists an integer $N$ such that for any $X/k\in\cC$, we have
$$N(\Br(X)/\Br_0(X))=0.$$
As a corollary to Lemma \ref{lemma:sujectivity_Brauer}, there exists an extension $K/k$ of degree $N$ such that
\begin{equation}\label{eqn:brsurj}
    \Br(X)/\Br_0(X)\longrightarrow \Br(X_K)/\Br_0(X_K)
\end{equation}
is surjective.

We now show that $X(\A_K)^{\Br}\neq\emptyset$. Let $\alpha\in\Br(X_K)$. By \eqref{eqn:brsurj}, there exists $\beta\in\Br(X)$ that maps to $\alpha$ modulo $\Br_0(X_K)$. Since the Brauer--Manin pairing is trivial on $\Br_0(X_K)$, we can assume without loss of generality that the image of $\beta$ is in fact $\alpha$. If $w\in\Omega_K$ is a place lying over $v\in\Omega_k$, then denote $e_w=[K_w:k_v]$. For any point $x_v\in X(k_v)$ with image $x_w\in X(K_w)$, we have the following commutative diagram.
$$\xymatrix{
    \Br(X_{k_v})\ar[r]^{x_v^*}\ar[d] & \Br(k_v)\ar[r]^{\inv_v}\ar[d] & \bQ/\bZ\ar[d]^{\times e_w} \\
    \Br(X_{K_w})\ar[r]^{x_w^*} & \Br(K_w)\ar[r]^{\inv_w} & \bQ/\bZ
}$$
Let $(x_v)\in X(\A_k)$ with image $(x_w)\in X(\A_K)$. Using the diagram above, we have
\begin{align*}
    \sum_{w\in \Omega_K}\inv_w \alpha(x_w)&=\sum_{v\in\Omega_k}\sum_{w\mid v}\inv_w \alpha(x_w)\\
    &=\sum_{v\in\Omega_k}\sum_{w\mid v}e_w\inv_v \beta(x_v)\\
    &=\sum_{v\in\Omega_k} N\inv_v\beta(x_v)=0.
\end{align*}
The last equality follows from the fact that $N\beta\in \Br_0(X)$. So $(x_w)\in X(\A_K)^{\Br}$ by the Brauer--Manin pairing \eqref{eqn:BMpairing}. Our hypothesis on BMWA says that $X(K)$ is dense in $X(\A_K)^{\Br}$. In the exact sequence
$$0\longrightarrow \Br_1(X_K)/\Br_0(X_K)[N]\longrightarrow \Br(X_K)/\Br_0(X_K)[N]\longrightarrow \Br(X_{\kbar})[N]\longrightarrow0,$$
we have that $\Br(X_{\kbar})[N]$ is finite by \cite{CTS19}*{Proposition~4.2.5} and $\Br_1(X)/\Br_0(X)\isom H^1(K,\Pic(X_{\kbar}))$ is finite since $\Pic(X_{\kbar})$ is torsion free of finite rank. Hence, $\Br(X_K)/\Br_0(X_K)$ is finite, which implies that $X(\A_K)^{\Br}$ is a non-empty open subset in $X(\A_K)$ in the adelic topology. Hence $X(K)$ must be Zariski dense. \end{proof}

\begin{cor}
Fix a primitive lattice $L\subset \Lambda_{K3}$ and integer $d>0$. Let $\cC=\{(X,k)\}$ be the collection of all K3 surfaces $X$ over number fields $k$ that are everywhere locally soluble, satisfying $\Pic(X_{\kbar})\isom L$ and $[k:\bQ]\leq d$. Assume Conjectures \ref{conj:UBEB} and \ref{conj:BMWA}. Then $\cC$ satisfies uniform potential density.
\end{cor}

\begin{proof}
Apply Theorem \ref{thm:bm} since K3 surfaces satisfy $H^1(X,\cO_X)=0$ and $\Pic(X_{\kbar})$ is torsion-free of bounded rank.
\end{proof}

\bigskip
\bibliography{PotentialDensityK3_bib}
\bibliographystyle{alpha}

\ContactInfo
\end{document}